\documentclass[10pt, a4paper]{article}
\pdfoutput=1

\usepackage[utf8]{inputenc}
\usepackage{epsfig,amsmath,latexsym,amssymb}
\usepackage{amsthm}
\usepackage[longnamesfirst,authoryear]{natbib}
\usepackage{graphicx}
\usepackage[hyphens]{url}
\usepackage[margin=1in]{geometry}
\usepackage{xcolor}
\usepackage[colorlinks,citecolor=blue,urlcolor=blue,filecolor=blue]{hyperref}
\usepackage{multirow}
\usepackage{enumitem}
\usepackage{lscape}
\usepackage{dsfont}
\usepackage[ruled,vlined]{algorithm2e}
\usepackage{algorithmic}
\usepackage{placeins}
\usepackage{authblk}

\setcitestyle{authoryear,open={(},close={)}}

\makeatletter
\newcommand{\subalign}[1]{%
  \vcenter{%
    \Let@ \restore@math@cr \default@tag
    \baselineskip\fontdimen10 \scriptfont\tw@
    \advance\baselineskip\fontdimen12 \scriptfont\tw@
    \lineskip\thr@@\fontdimen8 \scriptfont\thr@@
    \lineskiplimit\lineskip
    \ialign{\hfil$\m@th\scriptstyle##$&$\m@th\scriptstyle{}##$\hfil\crcr
      #1\crcr
    }%
  }%
}
\makeatother

\title{A Bonus-Malus Framework for Cyber Risk Insurance and Optimal Cybersecurity Provisioning}
\author[1]{Qikun Xiang}
\author[1]{Ariel Neufeld}
\author[2]{Gareth {W.} Peters}
\author[3]{Ido Nevat}
\author[4]{Anwitaman Datta}
\affil[1]{{\small{Division of Mathematical Sciences, Nanyang Technological University, Singapore}}}
\affil[2]{{\small{Department of Statistics and Applied Probability, University of California Santa Barbara, USA}}}
\affil[3]{{\small{TUMCREATE, Singapore}}}
\affil[4]{{\small{School of Computer Science and Engineering, Nanyang Technological University, Singapore}}}
\date{}

\usepackage{amsfonts,bbm}

\newcommand{\Tb}{\mathrm{b}}
\newcommand{\Tc}{\mathrm{c}}
\newcommand{\Td}{\mathrm{d}}
\newcommand{\Te}{\mathrm{e}}
\newcommand{\Tf}{\mathrm{f}}
\newcommand{\Tg}{\mathrm{g}}

\newcommand{\Ti}{\mathrm{i}}

\newcommand{\Tn}{\mathrm{n}}
\newcommand{\To}{\mathrm{o}}
\newcommand{\Tp}{\mathrm{p}}

\newcommand{\Tr}{\mathrm{r}}
\newcommand{\Ts}{\mathrm{s}}
\newcommand{\Tt}{\mathrm{t}}
\newcommand{\Tu}{\mathrm{u}}

\newcommand{\CA}{\mathcal{A}}
\newcommand{\CB}{\mathcal{B}}

\newcommand{\CD}{\mathcal{D}}

\newcommand{\CF}{\mathcal{F}}

\newcommand{\CI}{\mathcal{I}}

\newcommand{\CL}{\mathcal{L}}
\newcommand{\CM}{\mathcal{M}}

\newcommand{\CV}{\mathcal{V}}
\newcommand{\CW}{\mathcal{W}}

\newcommand{\FB}{\mathfrak{B}}

\newcommand{\BBV}{\mathbb{V}}

\newcommand{\R}{\mathbb{R}} 
\newcommand{\Z}{\mathbb{Z}} 
\newcommand{\N}{\mathbb{N}} 
\newcommand{\INDI}{\mathbbm{1}} 
\newcommand{\DIFF}{\Td} 
\newcommand{\DIFFX}[1]{\,\Td{#1}} 
\newcommand{\DIFFM}[2]{\,{#1}\!\left({#2}\right)\!}
\newcommand{\PROB}{\mathbb{P}} 
\newcommand{\EXP}{\mathbb{E}} 

\DeclareMathOperator*{\argmin}{arg\,min} 

\newtheorem{theorem}{Theorem}[section]

\newtheorem{lemma}[theorem]{Lemma}
\newtheorem{definition}[theorem]{Definition}
\newtheorem{remark}[theorem]{Remark}

\newtheorem{example}[theorem]{Example}

\providecommand{\keywords}[1]
{	
  {\small \textit{Keywords---}{#1}}
}

\newcommand{\COMMENT}[1]{\textcolor{gray}{\texttt{/*} #1 \texttt{*/}}}

\begin{document}
\maketitle

\begin{abstract}
The cyber risk insurance market is at a nascent stage of its development, even as the magnitude of cyber losses is significant and the rate of cyber loss events is increasing.
Existing cyber risk insurance products as well as academic studies have been focusing on classifying cyber loss events and developing models of these events, but little attention has been paid to proposing insurance risk transfer strategies that incentivise mitigation of cyber loss through adjusting the premium of the risk transfer product.
To address this important gap, we develop a Bonus-Malus model for cyber risk insurance.
Specifically, we propose a mathematical model of cyber risk insurance and cybersecurity provisioning supported with an efficient numerical algorithm based on dynamic programming.
Through a numerical experiment, we demonstrate how a properly designed cyber risk insurance contract with a Bonus-Malus system can resolve the issue of moral hazard and benefit the insurer.

\end{abstract}

\section{Introduction}

\subsection{The Ever-Increasing Threat of Cyber Crimes}
In recent years, the risk of cyber crimes has emerged as a top global concern. 
This is made evident by the world economic forum's annual global risk report \citep{wef2020}, which regularly places cyber attacks and theft of data in its ``Top 5 global risks in terms of likelihood''.
Over the years, the frequency and severity of cyber attacks have increased significantly globally, and they are projected to continue to increase in the future.
Recently, Cybersecurity Ventures estimated the cost of cyber crimes to rise to 10.5 trillion USD annually by 2025 \citep{morgan2020cybercrimecost}, up from a world economic forum estimate of 3 trillion USD for 2015.  
Furthermore, see a detailed descriptive analysis of cyber crimes in the business sector over the last 15 years in \citep{shevchenko2022nature}.

Cyber crime can be initiated by various actors, including malicious actors within institutions and external adversaries such as rogue nation states, hackers, and cyber terrorists.
These attacks come in a variety of forms, ranging from denial-of-service (DoS) attacks \citep{gupta2017taxonomy}, malware \citep{tailor2017comprehensive}, ransomware \citep{tailor2017comprehensive}, blackmail \citep{rid2012cyber}, extortion \citep{young1996cryptovirology}, and more \citep{craigen2014defining, husak2018survey}.
Such criminal activities are being perpetrated on a massive scale, hitting individuals as well as organisations. A large array of organisations worldwide are targeted by cyber attacks, including government agencies, universities, financial sectors, and private corporations. 
In particular, these attacks can affect important infrastructure units that play a key role in security and safety, such as emergency services and health care.
Such attacks have caused breaches and significant damages to those organisations, and often adversely affect downstream users of the compromised organisations and services.
The damages incurred include losses attributed to outcomes such as business interruption, loss of data, reduced reputation and trust of the organisation, legal liabilities, intellectual property theft, and potential for loss of life. 

The seriousness of cyber attacks has been reflected in the {U.S.} President's executive order on \textit{Strengthening the Cybersecurity of Federal Networks and Critical Infrastructure}, which calls for a cybersecurity framework that can ``support the cybersecurity risk management efforts of the owners and operators of the Nation’s critical infrastructure".
Cyber risk from a financial and insurance perspective has also been developed under international banking and insurance regulations. 
The Basel~III banking Accords cover cyber risk as a key component of Operational Risk capital modelling and adequacy, and the Solvency~II insurance regulations discuss the significance of an emerging cyber insurance threat that affects insurers as well as reinsurers.
For example, see an overview of cyber risk from a financial and insurance perspective in \citep{peters2018understanding} and \citep{malavasi2022cyber}.

\subsection{Challenges in the Current State of Cyber Risk Insurance} 
\label{ssec:challenges}

Although many security solutions have been developed and implemented in order to detect and prevent cyber attacks, achieving a complete security protection is not feasible \citep{lu2018managing}.
To address this problem, there is an increasing demand to develop the market for cyber risk insurance and to understand the structuring of insurance products that will facilitate risk transfer strategies in the context of cyber risk; see discussions in \citep{peters2018understanding, peters2017statistical, marotta2017cyber, bohme2010modeling} and the references therein. 

As one can see from surveys such as \citep{marotta2017cyber}, \citep{shetty2010competitive}, \citep{eling2020cyber}, and \citep{cremer2022cyber}, the scope of such a marketplace is still very much in its infancy, despite the fact that financial institutions rank cyber losses in their top three loss events systematically when reporting Operational Risk loss events under Basel II/III to national regulators. 
The reason that the cyber risk insurance market has yet to emerge with standardised products is largely due to the divergence in opinions as to how best to mitigate and reserve against these loss events. 
Below is a summary of three different perspectives on cyber risk insurance. 
\begin{enumerate}[label=(\roman*)]
\item From a technology perspective, it is common to attempt to mitigate such events in contrast to insurance or capital reserving; see discussions by \citet{bandyopadhyay2009managers}. 
\item From a financial risk perspective, practitioners often opt for Tier I capital reserving and avoid insurance products, as the capital reduction from Operational Risk insurance is capped under Basel regulations with a haircut of 20\%; see discussions by \citet{peters2011impact}. This disincentivises them to purchase insurance products that have excessive premiums. 
\item From the insurance industry's perspective, there is a lack of market standardisation on insurance contract specifications that would avoid excessive premiums to be charged when  bespoke insurance products are designed. 
\end{enumerate}

Moreover, as pointed out by recent studies and surveys by \citet{romannosky2019content, nurse2020data, roland2021multivariate, eling2020cyber}, and \citet{cremer2022cyber}, accurate modelling of cyber risk remains a major challenge largely due to the scarcity of accessible historical data on cyber loss events.
The lack of reliable historical cyber data results in high levels of uncertainty in the models, which ultimately lead to overpriced cyber risk cover \citep{kshetri2018economics}.
\citet{nurse2020data} identified the understanding of data collection and use in cyber risk insurance as a crucial aspect for future research and practice.
\citet{eling2020cyber} also mentioned that a critical challenge in understanding and measuring cyber risk is the accurate quantification of the effectiveness of cyber risk countermeasures, which demands more and higher quality information about cyber incidents to be shared.

\subsection{Our Contributions}

The three aforementioned perspectives are beginning to change and
the emergence of data sharing platforms \citep{eling2020cyber}, the potential introduction of data sharing mandates by relevant regulatory bodies, and advancements of modelling techniques in the academic world all point towards a promising outlook, where corporates, insurers/reinsurers, regulators, and cybersecurity experts/researchers can work more closely to achieve a better quantitative understanding of cyber incidents and cyber losses. 
As discussed in \citep{eiopa2018}, some insurers/reinsurers are beginning to improve  their pricing models for cyber risk by purchasing databases from external providers. 
Hence, we believe that it is a suitable time to revisit the perennial question of how best to set up an insurance marketplace for cyber loss events. 
In this paper, we introduce for the first time a Bonus-Malus framework for cyber risk insurance that provides IT-specific incentive mechanisms for encouraging sound IT governance and technology developments. 
Specifically, we explore a class of Bonus-Malus systems in which an insured enjoys a discount in the cost of risk transfer as a result of their upfront expenses in risk reduction in the form of self-mitigation measures; and like-wise an insurer benefits from encouraging risk reduction in their risk pools to provide competitive pricing of insurance premiums. 
Under this framework, we also demonstrate how to develop decision models under uncertainty.
Building this mutually beneficial mechanism into cyber risk insurance can hopefully attract more interests from potential customers and lead to better understandings of the effects of cybersecurity countermeasures on cyber incidents, which can help bridge the gap in product standardisation.

To illustrate this framework and the associated decision problems, let us consider an organisation which provides a service that is exposed to the threat of distributed denial-of-service (DDoS) attacks.
Such an organisation may opt to use network traffic filtering as well as a content distribution network, and may use multiple servers to balance the network load. 
As such, there are various countermeasures for mitigating the threats of DDoS attacks and ensuring the availability of the service, but each of such self-mitigation measures adds to the upfront costs of risk reduction. 
Therefore, the organisation needs to determine the quantum of its security (i.e., risk reduction) budget and distribute this across these self-mitigation measures. We argue that this budget for risk reduction needs to and can be determined in tandem with risk transfer decisions, i.e., the purchase of cyber risk insurance.
Moreover, a cyber risk insurer can incorporate incentive mechanisms in their insurance policies to encourage such risk reduction provision through offsetting the expenses of risk reduction by a discount in the pricing of the insurance premium.
This highlights the need for a comprehensive framework in which losses incurred by cyber threats can be realistically modelled and the rational decisions of organisations in the face of cyber threats can be quantitatively analysed.

Our main contributions are as follows: 
\begin{enumerate}
    \item We introduce a Bonus-Malus system to cyber risk insurance as a mechanism to provide incentive for the insured to adopt self-mitigation measures against cyber risk. 
\item We model cyber losses under the Loss Distribution Approach (LDA) framework where we include a mathematical description of cyber mitigation measures. 
Subsequently, decisions of the insured that pertain to cyber mitigation measures and cyber risk insurance contracts are jointly modelled within an optimal cybersecurity provisioning process under the stochastic optimal control framework.
\item We develop an efficient algorithm based on dynamic programming to accurately solve the stochastic optimal control problem, under the assumption that the loss severity follows a truncated version of the g-and-h distribution. We also formally prove the correctness of the proposed algorithm. 
\item We demonstrate through a numerical experiment that a properly designed cyber risk insurance contract with a Bonus-Malus system can resolve the issue of moral hazard, and can provide benefits for the insurer.
\end{enumerate}

Specifically, in our LDA model, the annual frequency and severity of cyber incidents are assumed to be independently distributed (this assumption is discussed and justified in Section~\ref{ssec:lossmodel}).
Moreover, rather than developing an explicit pricing model to determine the premium in the cyber risk insurance contract, our optimal cybersecurity provisioning process focuses on the insured's perspective, and it allows the insurers to evaluate and compare contract design options and to infer an optimal contract design from the insured's optimal responses.
This is discussed in Section~\ref{ssec:pricing}.

The rest of the paper is organised as follows. 
Section~\ref{sec:relatedwork} discusses related studies in the literature.
In Section~\ref{sec:model}, we introduce the mathematical model of cyber losses and cyber risk insurance with a Bonus-Malus system. In Section~\ref{sec:stocoptctr}, we present the optimal cybersecurity provisioning process and the dynamic programming algorithm. In Section~\ref{sec:g-and-h}, we introduce the g-and-h distribution and use it as the model for loss severity. We present results from the numerical experiment in Section~\ref{sec:exp}. Finally, Section~\ref{sec:conclusion} concludes the paper. 

\section{Related Work}
\label{sec:relatedwork}
Recently, many studies analysed cyber risk insurance from a technology perspective. 
Cybersecurity frameworks involving cyber risk insurance have been developed for specific IT systems, 
including computer networks \citep{fahrenwaldt2018pricing,xu2019cybersecurity}, 
heterogeneous wireless network \citep{lu2018cyber}, wireless cellular network \citep{lu2018managing}, plug-in electric vehicles \citep{hoang2017charging}, cloud computing \citep{chase2017scalable}, and fog computing \citep{feng2018evolving}. 

Some studies considered the interplay between self-mitigation measures (i.e., risk reduction) and cyber risk insurance, e.g., \citep{pal2010analyzing, pal2014will, pal2017security, khalili2018designing, dou2020insurance}. 
These studies investigated two important challenges in cyber risk insurance: risk interdependence and moral hazard. They found that in order to incentivise the insured to invest in self-mitigation measures, some form of contract discrimination, i.e., adjusting the insurance premium based on the insured's security investment, is necessary. 
\citet{yang2014security, schwartz2014cyber, zhang2017bi} investigated these challenges in a networked environment, where cyber attacks can spread between neighbouring nodes, further complicating these challenges.

There are also studies which took the insured's perspective and analysed the security provisioning process using dynamic models. 
\citet{chase2017scalable} developed a framework based on stochastic optimisation to jointly provision cyber risk insurance and cloud-based security services across multiple time periods in cloud computing applications. 
\citet{zhang2018optimal} modelled the decisions on self-protections of the insured by a Markov decision process and investigated the problem of insurance contract design.

A critical drawback of many of the existing studies is that they neglected the highly uncertain nature of losses incurred by cyber incidents. These studies relied on over-simplified assumptions, e.g., by modelling cyber loss as: a fixed amount \citep{pal2010analyzing,pal2014will,yang2014security,hoang2017charging,feng2018evolving,dou2020insurance}, random with finite support \citep{chase2017scalable,zhang2018optimal,lu2018managing}, or random with a simple parametric distribution \citep{zhang2017bi,khalili2018designing}. These assumptions limit the practicality of these studies, and their results remain conceptual and non-applicable to realistic insurance loss modelling under a classical Loss Distribution Approach (LDA) framework \citep{moscadelli2004modelling, peters2006bayesian, maillart2010heavy, shevchenko2013loss, eling2019actual, zeller2022comprehensive, dacorogna2023building, malavasi2022cyber}.
LDA is one of the most common modelling methods in the Advanced Measurement Approach (AMA) under the Basel II Accords. 
Under the LDA framework, the probability distributions of loss severity (i.e., the impact of a single loss event) and annual loss frequency are modelled and estimated separately. The aggregate annual loss is thus modelled by a compound distribution. 
For detailed discussions about LDA and a comparison with the Internal Measurement Approach, see \citep{frachot2001loss} and the references therein.
In addition, see \citep{dutta2006tale} for an evaluation of specific distributional assumptions in LDA and their impacts on the estimation of Operational Risk capital. 
In our study about the design of cyber risk insurance contracts, we have built LDA into our cyber loss model and we have also incorporated a quantitative model for the effects of self-mitigation measures on cyber losses (see Section~\ref{ssec:lossmodel}).

Moreover, many of the aforementioned studies do not take into account the interplay between the upfront costs of risk prevention, the consequent reduced risks, and the possibility to exploit this interplay to design practical cyber risk insurance products.
A review of cyber insurance product prospectus by major insurers, such as AIG's CyberEgde\footnote{\url{https://www.aig.com/business/insurance/cyber-insurance/}, accessed on 2020-12-10}, Allianz's Cyber Protect\footnote{\url{https://www.agcs.allianz.com/solutions/financial-lines-insurance/cyber-insurance.html}, accessed on 2020-12-10}, and Chubb's Cyber Enterprise Risk Management (Cyber ERM)\footnote{\url{https://www.chubb.com/us-en/business-insurance/cyber-enterprise-risk-management-cyber-erm.html}, \linebreak accessed on 2020-12-10} also indicates that the insurance products in the market have yet to explicitly factor in the benefits of upfront risk reduction, or to offset those costs against that of risk transfer.
For the first time, we introduce a Bonus-Malus system which is frequently used in vehicle insurance products to address this gap in cyber risk insurance product design. 
In vehicle insurance, Bonus-Malus systems are experience rating systems in which an insured who had one or more accidents is penalised by premium surcharges or \textit{maluses} and an insured who had a claim-free year is rewarded with premium discounts or \textit{bonuses} \citep{lemaire1995bonus, neuhaus1988bonus}; see, e.g., \citep{baione2002development, ragulina2017bonus, gomez2018multivariate, tzougas2018bonus} for discussions about the design and analysis of Bonus-Malus systems.
A key characteristic of Bonus-Malus systems is the bonus hunger mechanism \citep{holtan2001optimal, charpentier2017optimal, tzougas2018bonus}, i.e., under a Bonus-Malus system, an insured is willing to carry small losses themself in order to avoid premium surcharges in the future.
Compared to classical Bonus-Malus systems in the literature in which transitions of Bonus-Malus levels are solely determined by the presence/absence of a claim or claim frequencies (see, e.g., \citep{boucher2023bonus, gomezdeniz2016bivariate}), the Bonus-Malus system that we consider in our model incorporates flexible transition rules which can distinguish small claims from large claims.
This is a design recently investigated by, e.g., \citet{moumeesri2022bonus}.
Moreover, we also allow each Bonus-Malus level to have distinct premium, deductible, and maximum coverage (see Section~\ref{ssec:bonusmalus}). 
Furthermore, rather than using the traditional Bayesian methods \citep{moumeesri2022bonus} to determine Bonus-Malus premiums, we integrate our Bonus-Malus system into an optimal cybersecurity provisioning process (see Section~\ref{sec:stocoptctr}), in which the rational behaviours of the insured including decisions that pertain to cybersecurity mitigation measures and cyber risk insurance can be jointly quantitatively analysed. 
In particular, our optimal cybersecurity provisioning process captures the bonus hunger mechanism by allowing the insured to decide whether to make a claim at the end of each policy year.
This approach allows the insurer to tailor the transition rules in the Bonus-Malus contract and fine-tune the premiums to achieve their objectives, e.g., to maximum retention of customers.

\section{Cyber Risk Insurance Policy and Bonus-Malus System}
\label{sec:model}

Let us first present an overview of our cyber risk insurance model.
To begin, let us specify the frequency and severity model under the Loss Distribution Approach (LDA) that defines the financial loss process resulted from cyber loss events.
We consider $T\in\N$ consecutive years, and we assume that throughout each year $t$, the insured may suffer a random number ($N_t \in \N$) of cyber loss events arising from cyber attacks. Their loss amounts are denoted by $X^{(t)}_1,\ldots,X^{(t)}_{N_t}$. The aggregate annual cyber loss in year $t$ is therefore $\sum_{k=1}^{N_t}X^{(t)}_{k}$. 
The insured has several choices to attempt to mitigate these cyber loss events and reduce the risk, including enhancing the security and resilience of their IT infrastructure and reserving Tier~I capital to cover the incurred losses.
In regards to the internal IT infrastructure, it will be assumed that the insured has the option to adopt a self-mitigation measure that can reduce the severity of cyber loss events up to a fixed amount of loss. 
In addition, the insured can choose to purchase a cyber risk insurance policy which gives the insured the right to claim the aggregate cyber loss incurred, up to a maximum cap imposed by the insurance contract, in an agreed interval of time (typically annually), minus a deductible. 
Please refer to Table~\ref{tab:model-notations} for notations used in the cyber risk insurance model.

\begin{table}
\begin{center}
\caption{Notations in the cyber risk insurance model}
\label{tab:model-notations}
\vspace{0.5em}
\begin{tabular}{ll}
$T$ & number of policy years \\
$N_t$ & number of cyber loss events in year $t$ \\
$X^{(t)}_1,\ldots,X^{(t)}_{N_t}$ & loss amounts of the cyber loss events in year $t$  \\
$\CW$ & space representing annual loss frequency and severities \\ 
$(\Omega,\CF_T,\PROB,(\CF_t)_{t=0:T})$ & filtered probability space \\ 
$W_t$ & loss frequency and severities in year $t$ \\
$\psi_{N_t}(\cdot)$ & probability generating function of the loss frequency in year $t$ \\ 
$F_{X^{(t)}}(\cdot)$ & loss severity distribution function in year $t$ \\ 
$\CD$ & set representing self-mitigation measures \\
$\beta(d)$ & annual investment required by the self-mitigation measure $d\in\CD$ \\
$\gamma(d)$ & loss reduction effect of the self-mitigation measure $d\in\CD$ \\
$L(d,w)$ & aggregate annual cyber loss with the self-mitigation measure $d\in\CD$ \\
$\delta_{\Ti\Tn}(t)$ & initial sign-on fee of cyber risk insurance in year $t$ \\
$\delta_{\To\Tu\Tt}(t)$ & withdrawal penalty of cyber risk insurance in year $t$ \\
$\delta_{\Tr\Te}$ & re-activation penalty of cyber risk insurance \\
$\CB$ & set representing Bonus-Malus levels \\
$\CI$ & set representing states of the cyber risk insurance contract \\
$\CB\CM(\cdot,\,\cdot)$ & transition rules of the Bonus-Malus system \\
$\CB\CM_0(\cdot,\,\cdot)$ & transition rules of the Bonus-Malus system for inactive contract \\
$p^{\CB\CM}(b,t)$ & insurance premium in Bonus-Malus level $b$ in year $t$ \\
$l^{\CB\CM}_{\Td\Tt\Tb}(b,t)$ & insurance deductible in Bonus-Malus level $b$ in year $t$ \\
$l^{\CB\CM}_{\max}(b,t)$ & maximum insurance compensation in Bonus-Malus level $b$ in year $t$ \\
$\lambda^{\CB\CM}(b,t,l)$ & insurance compensation in Bonus-Malus level $b$ in year $t$ with loss $l$
\end{tabular}
\end{center}
\end{table}

\subsection{Loss Distribution Approach with Cyber Mitigation}
\label{ssec:lossmodel}

Let us define $\CW:=\bigcup_{n\in\Z_+}\{n\}\times\R^n_+$ to be the space representing all possible combinations of realisations of the number of cyber loss events per year (frequency) and individual loss amounts per event (severities).
Let $\FB(\CW):=\sigma\big(\bigcup_{n\in\Z_+}\{(n,B):B\in\FB(\R^n_+)\}\big)$ be a $\sigma$-algebra on $\CW$, where $\FB(\R^n_+)$ denotes the Borel subsets of $\R^n_+$. 
Let us consider the space $\Omega:=(\CW)^T=\underbrace{\CW\times\cdots\times\CW}_{T \text{ times}}$. For each $\omega=(w_1,\ldots,w_T)=\big(\big(n_1,x^{(1)}_1,\ldots,x^{(1)}_{n_1}\big),\ldots,$ $\big(n_T,x^{(T)}_1,\ldots,x^{(T)}_{n_T}\big)\big)\in\Omega$, we define $W_t(\omega):=w_t$ and $N_t(\omega):=n_t$ for $t=1,\ldots,T$. 
Let $\PROB_t$ be a probability measure on $(\CW,\FB(\CW))$, where the subscript $t$ indicates that it is a probability measure for the cyber loss events occurring in year $t$.
Let $\PROB:=\PROB_1\otimes\PROB_2\otimes\cdots\otimes\PROB_T$, and let $(\CF_t)_{t=0:T}$ be a filtration on $\Omega$, defined by 
$\CF_0:=\{\emptyset,\Omega\}$, $\CF_t:=\sigma((W_s)_{s=1:t})$. Then, $(\Omega,\CF_{T},\PROB,(\CF_t)_{t=0:T})$ is a filtered probability space. Under these definitions, $W_1,\ldots,W_T$ are independently distributed random variables.
Let $\psi_{N_t}(s):=\EXP[s^{N_t}]$ denote the probability generating function (pgf) of the loss frequency distribution in year $t$.
We assume that for $t=1,\ldots,T$, 
\begin{align}
\begin{split}
&\phantom{=}\;\;\PROB\Big[\Big\{\omega=\big(\big(n_1,x^{(1)}_1,\ldots,x^{(1)}_{n_1}\big),\ldots,\big(n_T,x^{(T)}_1,\ldots,x^{(T)}_{n_T}\big)\big):n_t=n,\;x^{(t)}_k\le z_k,\forall 1\le k\le n\Big\}\Big]\\
&=\PROB[N_t=n]\prod_{k=1}^n F_{X^{(t)}}(z_k),
\end{split}
\end{align}
where $F_{X^{(t)}}(\cdot)$ is the severity distribution function in year $t$. This implies that given the loss frequency in a year, the individual loss amounts in that year are independently and identically distributed (i.i.d.).
Even though it is well-known that dependence among cyber incidents (both intertemporal dependence and dependence across organisations) are present in practice \citep{biener2015insurability}, we neglect such dependence and assume independence in our model for simplicity since we are focusing on a single organisation. 
This is also a commonly adopted assumption in statistical analyses of cyber loss data; see, e.g., \citep{farkas2021cyber, malavasi2022cyber}.
We would like to remark that there are studies that take dependence into account in their models. 
For example, \citet{zangerle2023modelling} estimated the intertemporal dependence of cyber loss events. However, their models are limited to a small number of parameters due to scarcity of data.
\citet{peters2023cyber} analysed the dependence of cyber loss events across multiple companies and industry sectors, and highlighted the issue of parameter uncertainty in estimating the dependence structure and how it can translate into uncertainty about insurance premiums.
Moreover, we assume that the severity distributions all have finite expectation, i.e., $\int_{\R_+}|x|\DIFFM{F_{X^{(t)}}}{\DIFF x}<\infty$ for $t=1,\ldots,T$. 
For convenience, we write $W_t=\big(N_t,X^{(t)}_1,\ldots,X^{(t)}_{N_t}\big)$. 
Note that we allow the frequency and severity distributions to vary between years in order to capture the possible non-stationarity in cyber losses; for example, \citet{malavasi2022cyber} identified an increasing trend as well as an increase in variability in the number of cyber loss events between 2008 and 2014.
We use $X^{(t)}$ to denote a random variable that has the distribution function $F_{X^{(t)}}$.
We will discuss the properties of the severity distribution $F_{X^{(t)}}$ in Section~\ref{sec:g-and-h}.

\begin{remark}
Some probabilistic models of the severity of cyber incidents may have infinite expectation. For example, \citet{maillart2010heavy} analysed a dataset about personal identity theft and calibrated a power-law distribution to the number of records stolen in data breach incidents. They found that the calibrated model exhibits infinite expectation.
The analysis of a similar dataset by \citet{farkas2021cyber} using Generalised Pareto regression tree yields infinite conditional expectations for certain values of the covariates.
\citet{malavasi2022cyber} also found using a real dataset that cyber event severity distribution has infinite expectation. 
On the other hand, some studies of real cyber loss datasets concluded that the loss expectation is finite; see, e.g., \citep{dacorogna2023building}.
Moreover, \citet{peters2023cyber} remarked that the estimation of the tail behaviour of cyber loss severity depends heavily on data quality and sample size, and that significant model risk and parameter uncertainty may be present in the estimation, which can lead to insurance contracts being mispriced.
Even though the infiniteness of the loss expectation raises questions about the insurability of cyber risk, we argue that the finite expectation assumption here does not limit the practicality of our approach when we are interested in a single organisation.
From the insurer's perspective, the cyber risk insurance contract only covers a particular class of cyber loss events, and they can thus exclude all other event types from their analyses.
Moreover, cyber insurance contracts usually impose an upper limit on the compensation (see Section~\ref{ssec:cyberinsurance}), which guarantees the finiteness of the expected compensation.
From the insured's (i.e., the organisation's) perspective, there is an implicit upper bound on the monetary losses that can result from cyber incidents;
for example, the maximum annual monetary loss is bounded above by the valuation of the organisation.
\label{rmk:finite-expectation}
\end{remark}

Let us assume that there exists $D\in\N$ different self-mitigation measures, and the insured makes the decision to either adopt one of the self-mitigation measures or to not adopt any self-mitigation measure at the beginning of each year. The self-mitigation measure $d\in\CD:=\{0,1,\ldots,D\}$ requires an annual investment of $\beta(d)\in\R_+$ per year, and reduces the severity of each cyber loss by up to $\gamma(d)\in\R_+$, that is, the severity of a loss in year $t$ will be decreased from $X^{(t)}$ to\footnote{Throughout the paper, we use the following notations: $(x)^+:=\max\{x,0\}$, $x\vee y:=\max\{x,y\}$ and $x\wedge y:=\min\{x,y\}$.} $\left(X^{(t)}-\gamma(d)\right)^+$ with the adoption of the self-mitigation measure $d$. We assume that $\beta(0)=\gamma(0)=0$. 
Thus, if the insured decides to adopt the self-mitigation measure $d$ in a year, then the total loss suffered by the insured that year is given by
\begin{align}
L(d,w):=\sum_{k=1}^{n}\left(x_k-\gamma(d)\right)^+,
\label{eqn:yearlyloss}
\end{align}
when the corresponding loss frequency and severities that year are $w=(n,x_1,\ldots,x_n)\in\CW$.
Note that the decreased severity $\left(X^{(t)}-\gamma(d)\right)^+$ leaves the tail part of the severity distribution unchanged. 
This is a practically relevant assumption due to the right-skewed and heavy-tailed nature of cyber severities, that is, major cyber incidents causing severe losses are rare (see, e.g., \citep{zangerle2023modelling}).
Thus, it is reasonable to assume that the adoption of self-mitigation measures is effective against simple forms of cyber attacks that are more frequently occurring, and has little effect against severe and sophisticated attacks initiated by cyber terrorists.

\begin{remark}
One practical factor that may limit the applicability of this model is calibration, that is, given a self-mitigation measure~$d$, how one can reliably determine its effect $\gamma(d)$.
Achieving an accurate calibration of the cyber loss model is challenging in practice \citep{eling2020cyber}, due to the scarcity of publicly available cyber loss datasets (see the discussion in Section~\ref{ssec:challenges}).
Moreover, research studies focusing on the causal links between the IT security investments of an organisation and the cyber losses it suffers is rare \citep{woods2021sok}. 
Such empirical analyses are complicated by the fact that IT security is highly complex as it not only depends on technology but also strongly depends on organisational and human aspects \citep{herath2020organizational}. 
Recently, \citet{skarczinski2022more} analysed data from 493 firms about their most severe cyber incidents in the past 12 months that was sampled from a survey in 2018/19. 
They used the method of partial least square path modelling (PLS-PM) to analyse the causal effects of IT security-specific indicators including IT security budget, password requirements, regular backups, anti-virus software, regular patching, firewall, written information security policies, regular risk assessments/pentests, and failure simulations on the direct costs of cyber incidents. 
This study provides valuable insights about how the effects of cyber risk mitigation measures can be modelled and calibrated in practice. 
Therefore, one can expect that the increasing availability of open datasets about cyber incidents as well as the advancement of causal modelling techniques can ease the challenge of estimating the effects of self-mitigation measures in the future. 
\label{rmk:mitigation-effects}
\end{remark}

\subsection{Cyber Risk Insurance Policy}
\label{ssec:cyberinsurance}
Let us now consider a cyber risk insurance contract that lasts for $T$ years. 
At the beginning of each year, the insured decides whether to activate the contract. In the case that the contract has been activated in a previous year, this corresponds to the insured deciding whether to continue the contract. If the contract is activated, the insured pays a premium to the insurer at the start of the year, in exchange for insurance coverage throughout the year. If the insured decides to withdraw from the contract, they no longer pays the premium and the contract is deactivated so that the insured receives no coverage. 
We further assume that the insured pays the insurer an initial sign-on fee for fixed costs and contract origination the first time a contract is initiated, in addition to the premium, and that this amount varies deterministically over time and will be denoted by $\delta_{\Ti\Tn}(t)\ge 0$ in year $t$. This can be used to incentivise the insured to activate the contract early. 
Furthermore, we also assume that the insured pays the insurer a deterministic and time-dependent penalty, denoted by $\delta_{\To\Tu\Tt}(t)\ge0$, when they withdraws from the contract in year $t$. Once withdrawn, the insured may re-activate the contract in a later year with a fixed penalty $\delta_{\Tr\Te}\ge0$. 

Suppose that the aggregate cyber loss suffered by the insured in a year is $L$. At the end of the year, the insured decides whether to make a claim to the insurer. Once the claim is processed, the insured receives a payment of $(L-l_{\Td\Tt\Tb})^+\wedge l_{\max}$ from the insurer as compensation, that is, the insured covers the loss up to the deductible $l_{\Td\Tt\Tb}\ge 0$, and the insurer covers all of the remaining loss up to the maximum compensation $l_{\max}\ge0$. 

\subsection{Bonus-Malus System}
\label{ssec:bonusmalus}

We now introduce the Bonus-Malus system to cyber risk insurance contracts. Let us assume that there are $B\in\N$ Bonus-Malus levels in the contract, denoted by $\CB:=\{-\underline{B},\ldots,-1,0,1,\ldots,\overline{B}\}$, where $\underline{B}+\overline{B}+1=B$. 
Here, the level $0$ is the initial Bonus-Malus level, and the lower the Bonus-Malus level, the higher the experience rating.
Consequently, a negative Bonus-Malus level grants the insured premium discounts, while a positive Bonus-Malus level penalises the insured by premium surcharges.
At $t=0$, the insured starts in the initial Bonus-Malus level, denoted by $b_0=0$. At the end of the $t$-th year, given that the contract is still active, the insurer determines the Bonus-Malus level of the insured based on their previous level $b_{t-1}$ and the amount of insurance claim $C_t$ that was given out to the insured in the $t$-th year, that is, $b_{t}=\CB\CM(b_{t-1},C_t)$, where $\CB\CM:\CB\times\R_+\to\CB$ denotes the deterministic rules that are transparent to the insured at the signing of the contract. We make the assumption that $\CB\CM(b,C)$ is non-decreasing in $C$ for each $b\in\CB$. 
Even when the insured has withdrawn from the contract, we assume that their Bonus-Malus level is still updated annually. 
Concretely, let us define $\CI:=\{\Tn\To,\To\Tn,\To\Tf\Tf_1,\ldots,\To\Tf\Tf_{T}\}$ as the set of all possible states of the cyber risk insurance contract. In $\CI$, ``$\Tn\To$'' denotes that the contract has not been signed yet, ``$\To\Tn$'' denotes that the contract is active, and ``$\To\Tf\Tf_y$'' denotes that the contract is withdrawn where $y\in\Z_+$ is a counter variable that is updated annually as long as the insured does not re-activate the contract. 
Let $\CB\CM_{0}:\CB\times\CI\to\CB\times\CI$ be a deterministic transition function that represents the update rules after the insured withdraws from the contract. At the end of the $t$-th year, given that the contract is inactive, the insurer determines the Bonus-Malus level $b_t$ and the insurance state $i_t$ of the insured based on their Bonus-Malus level and the insurance state in the previous year, that is, $(b_t,i_t)=\CB\CM_0(b_{t-1},i_{t-1})$. Since no such update is possible before the insured activates the contract, it is required that $\CB\CM_{0}(b,\Tn\To)=(b,\Tn\To)$ for all $b\in\CB$. We will formally model the evolution of $(b_t,i_t)_{t=0:T}$ by a controlled stochastic process in Section~\ref{ssec:provisionprocess}.

With the addition of the Bonus-Malus system to the cyber risk insurance contract, we assume that the premium depends on both time and the current Bonus-Malus level of the insured, and is given by $p^{\CB\CM}(b,t)$, where $p^{\CB\CM}:\CB\times\{1,\ldots,T\}\to\R_+$ is a deterministic function that is increasing in the first argument. The deductible and the maximum compensation are also assumed to be dependent on both time and the Bonus-Malus level, and are given by deterministic functions $l_{\Td\Tt\Tb}^{\CB\CM}:\CB\times\{1,\ldots,T\}\to\R_+$ and $l_{\max}^{\CB\CM}:\CB\times\{1,\ldots,T\}\to\R_+$. We define the function $\lambda^{\CB\CM}:\CB\times\{1,\ldots,T\}\times\R_+\to\R_+$ by
\begin{align}
\lambda^{\CB\CM}(b,t,l):=(l-l_{\Td\Tt\Tb}^{\CB\CM}(b,t))^+\wedge l_{\max}^{\CB\CM}(b,t)
\end{align}
to simplify the notation for the claimable loss from the insurer.  

\begin{remark}
Our loss model and insurance model are particularly suitable for analysing cyber risk insurance, due to the nature of cyber loss events.
Firstly, the compound loss model in LDA is vital for capturing the low frequency and high impact nature of certain cyber loss events \citep{maillart2010heavy, biener2015insurability, eling2017data, eling2015modelling, eling2019actual}. 
In particular, this loss model allows us to use the highly flexible g-and-h distribution for the loss severity, in order to capture the right-skewness and the heavy-tail of cyber losses.
This provides adequate modelling flexibility so that one can tailor this loss model, particularly the skewness and the tail index of the loss distribution, according to specific kinds of cyber threats and different business lines. 
We will discuss the details of the g-and-h distribution in Section~\ref{sec:g-and-h}.
Secondly, as studies such as \citep{oughton2019stochastic, armenia2021dynamic} have discussed, despite that complete mitigation of cyber risk is impractical, prevention of crippling damages from cyber attacks, which often include secondary damages to downstream customers, can be achieved with relatively low cybersecurity expenses. 
For example, countermeasures such as the adoption of two-factor authentication and the regular update or reconfiguration of the software system \citep{pate2018cyber} can often result in significant reduction of the risk of intrusion; see also our discussion about their effects in Remark~\ref{rmk:mitigation-effects}. 
Moreover, the adoption of self-mitigation measures generates positive externalities since it improves the overall security of the cyberspace. 
Due to these factors, there is a strong motive for cyber risk insurers to offer incentive mechanisms in their insurance products to encourage such upfront risk reduction effort and to alleviate the problem of moral hazard.
Overall, the combination of this flexible loss model and the incentive mechanisms built into the Bonus-Malus system is novel in the context of cyber risk insurance. 
On the other hand, the independence assumption on the individual loss amounts in our model means that it does not account for the systemic aspect of cyber risk.
We would also like to remark that the general specifications of our cyber risk insurance model make it also applicable to other idiosyncratic risk types which can be partially prevented by self-mitigation measures. 
Despite that, our analyses and discussions about our model will be carried out in the context of cyber risk, and we will not generalise our model to other risk types in order not to obscure the main objective of this paper.
\label{rmk:cyber-specificity}
\end{remark}

\section{Optimal Cybersecurity Provisioning and Stochastic Optimal Control}
\label{sec:stocoptctr}

\subsection{Cybersecurity Provisioning Process}
\label{ssec:provisionprocess}
Now, having introduced the model for cyber losses and cyber risk insurance, we consider the problem of optimal cybersecurity provisioning from the insured's point of view. 
It is assumed that the cybersecurity provisioning process takes place for $T$ consecutive years (same as the length of the cyber risk insurance contract).
Before the first year, the state of the cyber risk insurance contract is initialised to ``$\Tn\To$'', i.e., $i_0=\Tn\To$, and the Bonus-Malus level is initialised to level~$0$, i.e., $b_0=0$.
Subsequently, each year $t\in\{1,\ldots,T\}$ consists of the three following stages:
\begin{enumerate}
\item \textbf{Provision Stage.} The insured decides: (i) the self-mitigation measure to adopt in this year, denoted by $d_t\in\CD$ and (ii) whether to activate/withdraw/re-activate the cyber risk insurance contract, denoted by $\iota_t\in\{0,1\}$. Then, depending on the decision $\iota_t$, a premium payment $p^{\CB\CM}(b_{t-1},t)$ and/or a sign-on fee $\delta_{\Ti\Tn}(t)$, a withdraw penalty $\delta_{\To\Tu\Tt}(t)$, or a re-activation penalty $\delta_{\Tr\Te}$ will be incurred, as stated in Section~\ref{ssec:cyberinsurance} and Section~\ref{ssec:bonusmalus}.

\item \textbf{Operation Stage.} The random cyber loss events and the corresponding cyber losses suffered by the insured in this year, denoted by $W_t$, is realised in this stage according to the model described in Section~\ref{ssec:lossmodel}. The aggregate cyber loss incurred is given by $L(d_t,W_t)$.

\item \textbf{Claim Stage.} If the insurance contract is active, i.e., $\iota_t=1$, the insured decides whether or not to make a claim, denoted by $j_t\in\{0,1\}$. In the case where a claim is made, i.e., $j_t=1$, the insured receives a compensation from the insurer according to the aggregate cyber loss that year, given by $\lambda^{\CB\CM}(b_{t-1},t,L(d_t,W_t))$. Subsequently, the cyber risk insurance contract state $i_t$ and the Bonus-Malus level $b_t$ are updated based on the decisions $\iota_t$ and $j_t$, as described in Section~\ref{ssec:bonusmalus}. 
Concretely, if $\iota_t=1$, then the insurance contract state will be ``$\To\Tn$'' and the Bonus-Malus level will be updated according to the function $\CB\CM(\cdot,\,\cdot)$, i.e., $i_t=\To\Tn$,  $b_t=\CB\CM\big(b_{t-1},j_t\lambda^{\CB\CM}(b_{t-1},t,L(d_t,W_t))\big)$.
If $\iota_t=0$, then the insurance contract state and the Bonus-Malus level will be updated according to the function $\CB\CM_0(\cdot,\,\cdot)$, i.e., $(b_t,i_t)=\CB\CM_0(b_{t-1},i_{t-1})$.
\end{enumerate}

\begin{table}
\begin{center}
\caption{Notations in the optimal cybersecurity provisioning process and stochastic optimal control}
\label{tab:dp-notations}
\vspace{0.5em}
\begin{tabular}{ll}
$d_t$ & decision of self-mitigation measure adopted in year $t$ \\
$\iota_t$ & decision to activate/withdraw/re-activate the insurance contract in year $t$ \\
$j_t$ & decision of whether to make a claim in year $t$  \\
$\Pi$ & set containing all admissible decision policies \\ 
$f_t(b,i,d,\iota,j,w)$ & state transition function in year $t$ in stochastic optimal control \\ 
$g_t(b,i,d,\iota,j,w)$ & cost function in year $t$ in stochastic optimal control \\
$(b^\pi_t,i^\pi_i)_{t=0:T}$ & controlled stochastic process representing the Bonus-Malus and insurance states \\ 
$e^{-r}$ & discount factor in stochastic optimal control \\
$V^\pi_t$ & expected discounted future cost at year $t$ in stochastic optimal control \\ 
$\mathbb{V}_0$ & the optimal value of the stochastic optimal control problem \\
$\CV_t(b,i)$ & value function in dynamic programming \\
$\widehat{d}_t(b,i)$ & one-stage optimal decision of self-mitigation measure in year $t$ \\
$\widehat{\iota}_t(b,i)$ & one-stage optimal decision of cyber risk insurance in year $t$ \\
$\widehat{j}_t(b,i,w)$ & one-stage optimal decision of insurance claim in year $t$ \\
$P^\star_t\big[(b,i)\rightarrow(b',i')\big]$ & transition kernel of the optimally controlled process \\
$\overline{P}^\star_t(b,i)$ & marginal state occupancy probability of the optimally controlled process \\
$\zeta^{(m)}_t(b,i,w)$ & a quantity of interest that depends on the state and the losses \\
$\overline{\zeta}^{(m)}_t$ & expected value of a quantity of interest \\
$\overline{Z}_{\zeta^{(m)}}$ & aggregate expected value of a quantity of interest
\end{tabular}
\end{center}
\end{table}

Next, in order to formally define an optimal decision policy, let us introduce the stochastic optimal control formulation of the cybersecurity provisioning process. Please refer to Table~\ref{tab:dp-notations} for notations used in the cybersecurity provisioning process and the stochastic optimal control formulation.
Let
\begin{align}
\begin{split}
\Pi:=\Big\{\pi=(d_t,\iota_t,&j_t)_{t=1:T}:d_t:\Omega\to\CD,\;\iota_t:\Omega\to\{0,1\} \text{ are }\CF_{t-1}\text{-measurable},\\
&j_t:\Omega\to\{0,1\} \text{ is }\CF_t\text{-measurable},\;\{\iota_t=0,j_t=1\}=\emptyset,\text{ for }t=1,\ldots,T\Big\}
\end{split}
\end{align}
denote the set of all admissible decision policies. The conditions in the above definition are explained as follows.
\begin{itemize}
\item The decisions $d_t$ and $\iota_t$ are made before observing $W_t$, hence may depend on all available information up to year $t-1$.
\item The decision $j_t$ is made after observing $W_t$, hence may depend on all available information up to year $t$.
\item The condition $\{\iota_t=0,j_t=1\}=\emptyset$ requires that the insured may only make a claim (i.e., $j_t=1$) when the insurance contract is activated (i.e., $\iota_t=1$) in year $t$. 
\end{itemize}
For $t=1,\ldots,T$, let $f_t:\CB\times\CI\times\CD\times\{0,1\}\times\{0,1\}\times\CW\to\CB\times\CI$ be the state transition function for each year $t$, given by
\begin{align}
\begin{split}
f_t(b,i,d,\iota,j,w):=\begin{cases}
\big(\CB\CM(b,j\lambda^{\CB\CM}(b,t,L(d,w))),\To\Tn\big) &\text{if }\iota=1,\\
\CB\CM_0(b,i) &\text{if }\iota=0,
\end{cases}
\end{split}
\label{eqn:socstate}
\end{align}
that is, $f_t(b,i,d,\iota,j,w)$ returns the Bonus-Malus level and the insurance state in year $t$ given that the Bonus-Malus level and the insurance state in year $t-1$ are $b$ and $i$, the decisions in year $t$ are $d$, $\iota$, $j$, and the cyber loss events in year $t$ are $w$.
For $t=1,\ldots,T$, let $g_t:\CB\times\CI\times\CD\times\{0,1\}\times\{0,1\}\times\CW\to\R_+$ be the cost function for each year $t$, given by 
\begin{align}
\begin{split}
g_t(b,i,d,\iota,j,w)&:=\beta(d)+\iota p^{\CB\CM}(b,t)+\delta_{\Ti\Tn}(t)\INDI_{\{i=\Tn\To,\iota=1\}}+\delta_{\To\Tu\Tt}(t)\INDI_{\{i=\To\Tn,\iota=0\}}\\
&\hspace{1.5cm}+\delta_{\Tr\Te}\INDI_{\{i\ne\To\Tn,i\ne\Tn\To,\iota=1\}}+L(d,w)-\iota j\lambda^{\CB\CM}(b,t,L(d,w)),
\end{split}
\label{eqn:soccost}
\end{align}
where $\beta(d)$ is the investment of adopting the self-mitigation measure $d$, $\iota p^{\CB\CM}(b)$ corresponds to the cyber risk insurance premium, $\delta_{\Ti\Tn}(t)\INDI_{\{i=\Tn\To,\iota=1\}}+\delta_{\To\Tu\Tt}(t)\INDI_{\{i=\To\Tn,\iota=0\}}+\delta_{\Tr\Te}\INDI_{\{i\ne\To\Tn,i\ne\Tn\To,\iota=1\}}$ corresponds to the sign-on/withdrawal/re-activation costs, $L(d,w)$ is the aggregate cyber loss, and $\iota j\lambda^{\CB\CM}(b,t,L(d,w))$ corresponds to the compensation from the insurer. 

Subsequently, for any decision policy $\pi=(d_t,\iota_t,j_t)_{t=1:T}\in\Pi$, let us define the $(\CF_t)_{t=0:T}$-adapted controlled stochastic process $\big(b^{\pi}_t,i^{\pi}_t\big)_{t=0:T}$ as follows:
\begin{align}
\begin{split}
(b^\pi_0,i^\pi_0)&:=(0,\Tn\To),\\
(b^\pi_t,i^\pi_t)&:=f_t(b^\pi_{t-1},i^\pi_{t-1},d_t,\iota_t,j_t,W_t)\quad\text{for }t=1,\ldots,T.
\end{split}
\label{eqn:markovprocess}
\end{align}
Then, $g_t(b^\pi_{t-1},i^\pi_{t-1},d_t,\iota_t,j_t,W_t)$ corresponds to the cybersecurity cost in year $t$. For $\pi\in\Pi$, and $t=1,\ldots,T$, define $V^{\pi}_t$ by
\begin{align}
V^\pi_t:=\EXP\Big[\textstyle\sum_{s=t+1}^Te^{-(s-t)r}g_s(b^\pi_{s-1},i^\pi_{s-1},d_s,\iota_s,j_s,W_s)\Big|\CF_{t}\Big],
\end{align}
where $0<e^{-r}\le 1$ is the discount factor and $V^\pi_T:=0$,
that is, $V^\pi_t$ is the expected value of the discounted total future cybersecurity cost viewed at the end of year $t$.

We assume that in the optimal cybersecurity provisioning process, the objective of the insured is to minimise the expected value of the discounted total cybersecurity cost, i.e., $V^\pi_0$, over all admissible decision policies $\pi\in\Pi$. This is formulated as the following finite horizon stochastic optimal control problem:
\begin{align}
\BBV_0:=\inf_{\pi\in\Pi}V^\pi_0=\inf_{\pi\in\Pi}\EXP\Big[\textstyle\sum_{t=1}^Te^{-tr}g_t(b^\pi_{t-1},i^\pi_{t-1},d_t,\iota_t,j_t,W_t)\Big]. 
\label{eqn:stocoptctr}
\end{align}

\subsection{Dynamic Programming Algorithm}
\label{ssec:dpalgo}
The stochastic optimal control problem in (\ref{eqn:stocoptctr}) can be solved efficiently via the dynamic programming algorithm. 
The dynamic programming algorithm iteratively solves one-stage optimisation problems while going backward in time.
In the following, let us denote the optimal values of these one-stage optimisation problems by $(\CV_t)_{t=0:T}$, and denote their corresponding optimisers by $(\widehat{d}_t,\widehat{\iota}_t,\widehat{j}_t)_{t=1:T}$.
For $t=T,T-1,\ldots,0$, let $\CV_{t}:\CB\times\CI\to\R_+$ be recursively defined as follows: for every $b\in\CB,i\in\CI$, let
\begin{align}
\begin{split}
\CV_T(b,i)&:=0,\\
\CV_{t-1}(b,i)&:=e^{-r}\min_{d\in\CD,\,\iota\in\{0,1\}}\bigg\{\EXP\bigg[\min_{j\in\{0,1\}}\Big\{g_{t}(b,i,d,\iota,j,W_t)+\CV_{t}\big(f_{t}(b,i,d,\iota,j,W_t)\big)\Big\}\bigg]\bigg\}\\
&\hspace{280pt}\text{for }t=T,T-1,\ldots,1.
\end{split}
\label{eqn:dpvalue}
\end{align}
Moreover, for $t=1,\ldots,T$, let $\widehat{d}_t:\CB\times\CI\to\CD$, $\widehat{\iota}_t:\CB\times\CI\to\{0,1\}$, and $\widehat{j}_t:\CB\times\CI\times\CW\to\{0,1\}$ be defined as follows: for every $b\in\CB,i\in\CI$, let
\begin{align}
\big(\widehat{d}_t(b,i),\widehat{\iota}_t(b,i)\big)&\in \argmin_{d\in\CD,\,\iota\in\{0,1\}}\bigg\{\EXP\bigg[\min_{j\in\{0,1\}}\Big\{g_{t}(b,i,d,\iota,j,W_t)+\CV_{t}\big(f_{t}(b,i,d,\iota,j,W_t)\big)\Big\}\bigg]\bigg\},\label{eqn:dpconcept-d-iota}\\
\widehat{j}_t(b,i,w)&:=\begin{cases}
1 & \text{if }\widehat{\iota}_t(b,i)=1,\;\Big[g_{t}(b,i,\widehat{d}_t(b,i),1,1,w)+\CV_{t}\big(f_{t}(b,i,\widehat{d}_t(b,i),1,1,w)\big)\Big] \\
&\hspace{2.5cm}< \Big[g_{t}(b,i,\widehat{d}_t(b,i),1,0,w)+\CV_{t}\big(f_{t}(b,i,\widehat{d}_t(b,i),1,0,w)\big)\Big],\\
 0 & \text{otherwise}.
\end{cases}
\label{eqn:dpconcept-j}
\end{align}
Thus, for $t=0,\ldots,T$, $\CV_t(b,i)$ corresponds to the optimal total discounted future cybersecurity cost when viewing the optimal cybersecurity provisioning process at the end of year $t$ when the Bonus-Malus level and the insurance state are $b$ and $i$. 
Observe from (\ref{eqn:dpconcept-d-iota}) and (\ref{eqn:dpconcept-j}) that $\big(\widehat{d}_t(b,i),\widehat{\iota}_t(b,i)\big)$ solves the minimisation outside the expectation in (\ref{eqn:dpvalue}), while $\widehat{j}_t(b,i,W_t)$ solves the minimisation inside the expectation in (\ref{eqn:dpvalue}) with $(d,i)\leftarrow\big(\widehat{d}_t(b,i),\widehat{\iota}_t(b,i)\big)$.
Since the minimisation over $j\in\{0,1\}$ in (\ref{eqn:dpvalue}) is inside the expectation, the decision $j$ is made after observing the realisation of the cyber loss events $W_t$, and thus the optimiser $\widehat{j}_t$ is allowed to depend on $W_t$. On the other hand, since the minimisation over $d\in\CD$, $\iota\in\{0,1\}$ is outside the expectation, these decisions need to be made before observing $W_t$, and thus $\widehat{d}_t$ and $\widehat{\iota}_t$ cannot depend on $W_t$.
In the first iteration of the dynamic programming algorithm, one computes $\big(\widehat{d}_T(b,i),\widehat{\iota}_T(b,i),\widehat{j}_T(b,i,w)\big)$ from (\ref{eqn:dpconcept-d-iota}) and (\ref{eqn:dpconcept-j}) for every $(b,i,w)\in\CB\times\CI\times\CW$. Moreover, one computes the values of $(\CV_{T-1}(b,i))_{(b,i)\in\CB\times\CI}$ from (\ref{eqn:dpvalue}).
In subsequent iterations of the dynamic programming algorithm, one iterates through $t=T-1,T-2,\ldots,1$, where, in each iteration, one considers the one-stage optimisation problem in year $t$ given the values $(\CV_t(b,i))_{(b,i)\in\CB\times\CI}$.
Specifically, one computes $\big(\widehat{d}_t(b,i),\widehat{\iota}_t(b,i),\widehat{j}_t(b,i,w)\big)$ from (\ref{eqn:dpconcept-d-iota}) and (\ref{eqn:dpconcept-j}) for every $(b,i,w)\in\CB\times\CI\times\CW$, and computes the values of $(\CV_{t-1}(b,i))_{(b,i)\in\CB\times\CI}$ from (\ref{eqn:dpvalue}).
The following theorem shows how one can construct an optimal decision policy $\pi^\star$ for the stochastic optimal control problem (\ref{eqn:stocoptctr}) by piecing together the one-stage optimisers $\big(\widehat{d}_t,\widehat{\iota}_t,\widehat{j}_t\big)_{t=1:T}$ obtained via dynamic programming.

\begin{theorem}
Let the functions $(\CV_t)_{t=0:T}$ be defined as in (\ref{eqn:dpvalue}), let the functions $(\widehat{d}_t,\widehat{\iota}_t,\widehat{j}_t)_{t=1:T}$ be defined as in (\ref{eqn:dpconcept-d-iota}) and (\ref{eqn:dpconcept-j}), and let $\pi^\star=(d^\star_t,\iota^\star_t,j^\star_t)_{t=1:T}\in\Pi$ be recursively defined as follows: 
\begin{align}
\begin{split}
&\hspace{-0.75cm}\big(b^{\pi^\star}_{0},i^{\pi^\star}_{0}\big):=(0,\Tn\To),\\
\text{for }t&=1,\ldots,T,\text{ let:}\\
&d^\star_t:=\widehat{d}_t\big(b_{t-1}^{\pi^\star},i_{t-1}^{\pi^\star}\big),\quad\quad \iota^\star_t:=\widehat{\iota}_t\big(b_{t-1}^{\pi^\star},i_{t-1}^{\pi^\star}\big),\quad\quad j^\star_t:=\widehat{j}_t\big(b_{t-1}^{\pi^\star},i_{t-1}^{\pi^\star},W_t\big).
\end{split}
\label{eqn:dpconcept-opt-process}
\end{align}
Then, the following holds:
\begin{align}
\CV_0(0,\Tn\To)=V_0^{\pi^\star}=\BBV_0.
\end{align}
\label{thm:dp}
\end{theorem}
\begin{proof}
See Appendix~\ref{apx:proofs}. 
\end{proof}

Besides an optimal decision policy $\pi^\star$ constructed in Theorem~\ref{thm:dp}, we are often also interested in other quantities related to the optimal cybersecurity provisioning process, such as the transition kernels and the marginal state occupancy probabilities of the process $\big(b^{\pi^\star}_t,i^{\pi^\star}_t\big)_{t=0:T}$ as well as other quantities of interest.
Specifically, given an optimal decision policy $\pi^\star$ constructed in Theorem~\ref{thm:dp}, its associated optimally controlled process $\big(b^{\pi^\star}_t,i^{\pi^\star}_t\big)_{t=0:T}$ is a discrete-time time-inhomogeneous Markov chain, and we are interested in its transition kernels $\big(P^\star_t\big)_{t=1:T}$, where
\begin{align}
\begin{split}
P^\star_t\big[(b,i)\rightarrow(b',i')\big]&:=\PROB\Big[\big(b^{\pi^\star}_{t},i^{\pi^\star}_{t}\big)=(b',i')\Big|\big(b^{\pi^\star}_{t-1},i^{\pi^\star}_{t-1}\big)=(b,i)\Big] \\
&\hspace{160pt}\forall (b,i),(b',i')\in\CB\times\CI,\;\text{for }t=1,\ldots,T,
\end{split}
\label{eqn:dpalgo-kernel}
\end{align}
as well as its marginal state occupancy probabilities $\big(\overline{P}^\star_t\big)_{t=0:T}$, where
\begin{align}
\overline{P}^\star_t(b,i):=\PROB\Big[\big(b^{\pi^\star}_{t},i^{\pi^\star}_{t}\big)=(b,i)\Big]\qquad\forall (b,i)\in\CB\times\CI,\;\text{for }t=0,\ldots,T.
\label{eqn:dpalgo-marg}
\end{align}
Using these probabilities, we are able to compute important quantities that will aid us in the design and evaluation of cyber risk insurance contracts, such as the expected number of years the insured will spend in each of the Bonus-Malus levels.
Moreover, we are often interested in computing the expected values of other quantities of interest which can be used as evaluation metrics for the cyber risk insurance contract.
Formally, suppose that we are interested in $M$ quantities, and let $\zeta^{(m)}_t:\CB\times\CI\times\R_+\to\R$ for $t=1,\ldots,T$, $m=1,\ldots,M$ satisfy
\begin{align}
\EXP\big[\big|\zeta^{(m)}_t(b,i,W_t)\big|\big]<\infty\qquad \forall (b,i)\in\CB\times\CI,\;\text{for }t=1,\ldots,T,\;\text{for }m=1,\ldots,M.
\label{eqn:qoi}
\end{align}
Let us define
\begin{align}
\overline{\zeta}^{(m)}_t&:=\EXP\big[\zeta^{(m)}_t\big(b^{\pi^\star}_{t-1},i^{\pi^\star}_{t-1},W_t\big)\big]\qquad\text{for }t=1,\ldots,T,\;\text{for }m=1,\ldots,M,\label{eqn:qoi-expval}\\
\overline{Z}_{\zeta^{(m)}}&:=\sum_{t=1}^T\overline{\zeta}_t^{(m)}\hspace{139pt}\qquad\text{for }m=1,\ldots,M.\label{eqn:qoi-aggregate}
\end{align}
Thus, for $m=1,\ldots,M$, $\zeta^{(m)}_t\big(b^{\pi^\star}_{t-1},i^{\pi^\star}_{t-1},W_t\big)$ represents a quantity of interest in year $t$ that (possibly) depends on the Bonus-Malus level and the insurance state at the end of year $t-1$, i.e., $b^{\pi^\star}_{t-1}$ and $i^{\pi^\star}_{t-1}$, and the losses in year $t$, i.e., $W_t$. 
Moreover, for $m=1,\ldots,M$, $\overline{\zeta}^{(m)}_t$ denotes the expected value of the quantity of interest $\zeta^{(m)}_t\big(b^{\pi^\star}_{t-1},i^{\pi^\star}_{t-1},W_t\big)$ in year $t$, and $\overline{Z}_{\zeta^{(m)}}$ denotes the aggregate expected value of this quantity of interest over $T$ years. 
Due to the independence between $\CF_{t-1}$ and $\sigma(W_t)$, one observes that $(\overline{\zeta}^{(m)}_t)_{t=1:T}$ and hence also $\overline{Z}_{\zeta^{(m)}}$ can be computed based on the marginal state occupancy probabilities $\big(\overline{P}^\star_t\big)_{t=0:T}$ defined in (\ref{eqn:dpalgo-marg}). Example~\ref{exp:quantities} below provides concrete examples of quantities of interest that can aid the design and analysis of cyber risk insurance contracts.

\begin{example}[Quantities of interest]
In this example, we let $M=D+5$, and define the following quantities of interest. 
\begin{enumerate}[label=(\roman*)]
\item \label{sexp:quantities-mitiadoption}For every $d\in\CD:=\{0,1,\ldots,D\}$, let $\zeta^{(d+1)}_{t}(b,i,w)=\INDI_{\{\widehat{d}_t(b,i)=d\}}$. Then, we have
$\overline{\zeta}^{(d+1)}_{t}=$ \linebreak$\PROB\big[\widehat{d}_t\big(b^{\pi^\star}_{t-1},i^{\pi^\star}_{t-1}\big)=d\big]$, which corresponds to the probability that the self-mitigation measure $d$ is adopted in year $t$ under the decision policy $\pi^\star$.
\item Let $\zeta^{(D+2)}_{t}(b,i,w)=e^{-tr}\beta(\widehat{d}_t(b,i))$. Then, $\overline{Z}_{\zeta^{(D+2)}}$ corresponds to the expected value of the discounted total self-mitigation expenses. 
\item \label{sexp:quantities-premium}Let $\zeta^{(D+3)}_{t}(b,i,w)=e^{-tr}\big(\widehat{\iota}_t(b,i)\big(p^{\CB\CM}(b)+\delta_{\Ti\Tn}(t)\INDI_{\{i=\Tn\To\}}+\delta_{\Tr\Te}\INDI_{\{i\ne\Tn\To,i\ne\To\Tn\}}\big)+(1-\widehat{\iota}_t(b,i))\delta_{\To\Tu\Tt}(t)\INDI_{\{i=\To\Tn\}}\big)$.
Then, $\overline{Z}_{\zeta^{(D+3)}}$ corresponds to the expected value of the discounted total payment from the insured to the insurer. 
\item \label{sexp:quantities-mitieffect}Let $\zeta^{(D+4)}_{t}\big(b,i,w=(n,x_1,\ldots,x_n)\big)=e^{-tr}\left(\sum_{k=1}^{n}x_k-L(\widehat{d}_t(b,i),w)\right)$. Then, $\overline{Z}_{\zeta^{(D+4)}}$ corresponds to the expected value of the discounted total loss that is prevented by the self-mitigation measures. 
\item \label{sexp:quantities-benefit}Let $\zeta^{(D+5)}_{t}(b,i,w)=e^{-tr}\Big(\widehat{j}_t(b,i,w)\lambda^{\CB\CM}(b,t,L(\widehat{d}_t(b,i),w))\Big)$, then $\overline{Z}_{\zeta^{(D+5)}}$ corresponds to the expected value of the discounted total insurance compensation the insured receives. 
\end{enumerate}
\label{exp:quantities} 
\end{example}

As a consequence of Theorem~\ref{thm:dp}, let us now introduce a concrete algorithm based on the dynamic programming principle to solve the stochastic optimal control problem (\ref{eqn:stocoptctr}), which is presented in Algorithm~\ref{alg:dpconcrete}. 
In addition to an optimal decision policy $\pi^\star$, Algorithm~\ref{alg:dpconcrete} also computes the state transition kernels $\big(P^\star_t\big)_{t=1:T}$ and the marginal state occupancy probabilities $\big(\overline{P}^\star_t\big)_{t=0:T}$ of the optimally controlled process $\big(b^{\pi^\star}_t,i^{\pi^\star}_t\big)_{t=0:T}$, defined in (\ref{eqn:dpalgo-kernel}) and (\ref{eqn:dpalgo-marg}).
Moreover, given $M$ quantities of interest $\big(\zeta^{(m)}(\cdot,\cdot,\cdot)\big)_{m=1:M}$ as input, Algorithm~\ref{alg:dpconcrete} also computes their expected values defined in (\ref{eqn:qoi-expval}) as well as their aggregate expected values defined in (\ref{eqn:qoi-aggregate}). 
Remark~\ref{rmk:dpalgo} provides detailed explanations of the inputs and outputs of Algorithm~\ref{alg:dpconcrete} as well as explanations of some of the lines in Algorithm~\ref{alg:dpconcrete}. 
Theorem~\ref{thm:dpalgo} shows the correctness of Algorithm~\ref{alg:dpconcrete}. 
Moreover, Remark~\ref{rmk:comp} discusses the computational tractability of Algorithm~\ref{alg:dpconcrete}.

\begin{algorithm}[p]
\KwIn{$\CB$, $\CI$, $\CD$, $\CB\CM$, $\CB\CM_{0}$, $p^{\CB\CM}$, $l^{\CB\CM}_{\Td\Tt\Tb}$, $l^{\CB\CM}_{\max}$, $\beta(\cdot)$, $\gamma(\cdot)$, $\delta_{\Ti\Tn}(\cdot)$, $\delta_{\To\Tu\Tt}(\cdot)$, $\delta_{\Tr\Te}$, $r$, $\big(\zeta^{(m)}_t(\cdot,\cdot,\cdot)\big)_{m=1:M}$}
\KwOut{$\CV_0(0,\Tn\To)$, $\pi^{\star}$, $\big(P^\star_t\big)_{t=1:T}$, $\big(\overline{P}^\star_t\big)_{t=0:T}$, $\big(\overline{\zeta}_{t}^{(m)}\big)_{t=1:T,m=1:M}$, $\big(\overline{Z}_{\zeta^{(m)}}\big)_{m=1:M}$}
\nl $\CV_T(b,i)\leftarrow0$ for all $b\in\CB,i\in\CI$. \label{alglin:dp-value-init} \\
\nl \For{$t=T,T-1,\ldots,1$}{
\nl \label{alglin:dp-forloop-b}\For{$b\in\CB$}{
\nl $\underline{b}\leftarrow\CB\CM(b,0)$, $\overline{b}\leftarrow\max\{\CB\CM(b,c):c\in\R_+\}$. \COMMENT{$\underline{b}$ and $\overline{b}$ are the lowest and highest possible Bonus-Malus levels reachable from level $b$} \label{alglin:dp-b-b-def} \\
\nl \For{$\underline{b}\le b'\le\overline{b}$}{
\nl $\alpha_t(b,b')\leftarrow \CV_{t}(b',\To\Tn)-\CV_{t}(\underline{b},\To\Tn)$. \COMMENT{temporary value to simplify computation} \label{alglin:dp-alpha-def} \\
\nl $\CL_t(b,b')\leftarrow\left\{c\in\R_+:\CB\CM(b,c)=b',c>\alpha_t(b,b')\right\}$. \COMMENT{the set of potential insurance compensation amounts such that the optimal decision is to make a claim and the updated Bonus-Malus level will be $b'$; see also Remark~\ref{rmk:dpalgo}}\label{alglin:dp-L-def} \\
}
\nl \label{alglin:dp-forloop-d}\For{$i\in\CI$, $d\in\CD$}{
\nl $H_t(b,i,d,1)\leftarrow \CV_t(\underline{b},\To\Tn)-\sum_{\underline{b}\le b'\le\overline{b}}\EXP\bigg[\INDI_{\{\CB\CM(b,\lambda^{\CB\CM}(b,t,L(d,W_t)))=b'\}}\Big(\lambda^{\CB\CM}(b,t,L(d,W_t))-\alpha_t(b,b')\Big)^+\bigg]$.\label{alglin:dp-H-def1} \\
\nl $H_t(b,i,d,0)\leftarrow \CV_t\big(\CB\CM_0(b,i)\big)$. \COMMENT{$H_t(b,i,d,\iota)$ for $\iota\in\{0,1\}$ are temporary values to simplify the one-stage optimisation problem in Line~\ref{alglin:dp-d-iota-def}; see also Remark~\ref{rmk:dpalgo}}\label{alglin:dp-H-def2} \\
}
\nl $(\widehat{d}_t(b,i),\widehat{\iota}_t(b,i))\leftarrow\argmin_{d\in\CD,\,\iota\in\{0,1\}} \bigg\{\beta(d)+\iota p^{\CB\CM}(b,t)+\delta_{\Ti\Tn}(t)\INDI_{\{i=\Tn\To,\iota=1\}}+\delta_{\To\Tu\Tt}(t)\INDI_{\{i=\To\Tn,\iota=0\}}+\delta_{\Tr\Te}\INDI_{\{i\ne\To\Tn,i\ne\Tn\To,\iota=1\}}+\EXP\big[L(d,W_t)\big]+H_t(b,i,d,\iota)\bigg\}$.\label{alglin:dp-d-iota-def} \\
\nl $\widehat{j}_t(b,i,w)\leftarrow\INDI_{\{\widehat{\iota}_t(b,i)=1\}}\INDI_{\bigcup_{\underline{b}\le b'\le\overline{b}}\CL_t(b,b')}\big(\lambda^{\CB\CM}(b,t,L(\widehat{d}_t(b,i),w))\big)$.\label{alglin:dp-j-def} \\
\nl $\CV_{t-1}(b,i)\leftarrow e^{-r}\min_{d\in\CD,\,\iota\in\{0,1\}} \bigg\{\beta(d)+\iota p^{\CB\CM}(b,t)+\delta_{\Ti\Tn}(t)\INDI_{\{i=\Tn\To,\iota=1\}}+\delta_{\To\Tu\Tt}(t)\INDI_{\{i=\To\Tn,\iota=0\}}+\delta_{\Tr\Te}\INDI_{\{i\ne\To\Tn,i\ne\Tn\To,\iota=1\}}+\EXP\big[L(d,W_t)\big]+H_t(b,i,d,\iota)\bigg\}$.\label{alglin:dp-value-update} \\
\nl $P^\star_t\big[(b,i)\rightarrow(b',i')\big]\leftarrow 0$ for all $(b',i')\in\CB\times\CI$. \COMMENT{initialise all transition probabilities \\ to~0} \\
\nl \If{$\widehat{\iota}_t(b,i)=1$}{
\nl \For{$\underline{b}<b'\le\overline{b}$}{
\nl $P^\star_t\big[(b,i)\rightarrow(b',\To\Tn)\big]\leftarrow {\PROB\Big[\lambda^{\CB\CM}(b,t,L(\widehat{d}_t(b,i),W_t))\in\CL_t(b,b')\Big]}$. \COMMENT{compute the transition probability to level $b'$}\label{alglin:dp-kern1} \\
}
\nl $P^\star_t\big[(b,i)\rightarrow(\underline{b},\To\Tn)\big]\leftarrow 1-\sum_{\underline{b}<b'\le\overline{b}}P^\star_t\big[(b,i)\rightarrow(b',\To\Tn)\big]$. \label{alglin:dp-kern2} \COMMENT{compute the transition probability to the lowest level}\\
}\nl \Else{
\nl $P^\star_t\big[(b,i)\rightarrow\CB\CM_0(b,i)\big]\leftarrow 1$. \COMMENT{when the optimal decision is to not purchase insurance, the transition is deterministic} \label{alglin:dp-kern3} \\
}
}
}
\nl $\overline{P}_{0}^{\star}(0,\Tn\To)\leftarrow 1$, $\overline{P}_{0}^{\star}(b,i)\leftarrow 0$ for all $(b,i)\ne(0,\Tn\To)$. \label{alglin:dp-marg-init} \\
\nl \For{$t=1,2,\ldots,T$}{
\nl For all $(b,i)\in\CB\times\CI$, $\overline{P}_{t}^{\star}(b,i)\leftarrow\sum_{(b',i')\in\CB\times\CI}P^\star_t[(b',i')\rightarrow(b,i)]\overline{P}_{t-1}^{\star}(b',i')$. \label{alglin:dp-marg-iter} \\
\nl For $m=1,\ldots,M$, $\overline{\zeta}^{(m)}_t\leftarrow\sum_{(b,i)\in\CB\times\CI}\EXP[\zeta^{(m)}_t(b,i,W_t)]\overline{P}_{t-1}^{\star}(b,i)$. \label{alglin:dp-qoi} \\
}
\nl For $m=1,\ldots,M$, $\overline{Z}_{\zeta^{(m)}}\leftarrow\sum_{t=1}^T\overline{\zeta}_t^{(m)}$. \label{alglin:dp-qoi-aggregate} \\
\nl Define $\pi^\star=(d^\star_t,\iota^\star_t,j^\star_t)_{t=1:T}$ by (\ref{eqn:dpconcept-opt-process}).\\
\nl \Return $\CV_0(0,\Tn\To)$, $\pi^{\star}$, $\big(P^\star_t\big)_{t=1:T}$, $\big(\overline{P}^\star_t\big)_{t=0:T}$, $\big(\overline{\zeta}_{t}^{(m)}\big)_{t=1:T,m=1:M}$, $\big(\overline{Z}_{\zeta^{(m)}}\big)_{m=1:M}$. 
    \caption{{\bf Dynamic Programming for Optimal Cybersecurity Provisioning}}
    \label{alg:dpconcrete}
\end{algorithm}

\begin{remark}[Details about Algorithm~\ref{alg:dpconcrete}]
The inputs of Algorithm~\ref{alg:dpconcrete} are explained as follows.
\begin{itemize}
\item $\CB$, $\CI$, $\CD$, $\CB\CM$, $\CB\CM_{0}$, $p^{\CB\CM}$, $l^{\CB\CM}_{\Td\Tt\Tb}$, $l^{\CB\CM}_{\max}$, $\beta(\cdot)$, $\gamma(\cdot)$, $\delta_{\Ti\Tn}(\cdot)$, $\delta_{\To\Tu\Tt}(\cdot)$, $\delta_{\Tr\Te}$ are specified in Section~\ref{ssec:lossmodel}, Section~\ref{ssec:cyberinsurance}, and Section~\ref{ssec:bonusmalus}.
\item $r\ge0$ specifies the discount factor $e^{-r}$ in the objective of the optimal cybersecurity provisioning process.
\item $\big(\zeta^{(m)}_t(\cdot,\cdot,\cdot)\big)_{m=1:M}$ specify $M$ quantities of interest that satisfy (\ref{eqn:qoi}) whose expected values (\ref{eqn:qoi-expval}) and aggregate expected values (\ref{eqn:qoi-aggregate}) will be computed in Algorithm~\ref{alg:dpconcrete}. Some examples of quantities of interest are shown in Example~\ref{exp:quantities}. 
\end{itemize}
The outputs of Algorithm~\ref{alg:dpconcrete} are explained as follows.
\begin{itemize}
\item $\pi^{\star}$ is the optimal decision policy computed by the dynamic programming algorithm and $\CV_0(0,\Tn\To)$ is equal to the optimal value $\BBV_0$ of the stochastic optimal control problem. The optimality of $\pi^\star$ is shown in Theorem~\ref{thm:dpalgo}\ref{sthm:dpalgo1}. 
\item $\big(P^\star_t\big)_{t=1:T}$ are the transition kernels of the discrete-time time-inhomogeneous Markov chain \linebreak$\big(b^{\pi^\star}_t,i^{\pi^\star}_t\big)_{t=0:T}$ that are defined in (\ref{eqn:dpalgo-kernel}). $\big(\overline{P}^\star_t\big)_{t=0:T}$ are the marginal state occupancy probabilities of $\big(b^{\pi^\star}_t,i^{\pi^\star}_t\big)_{t=0:T}$ defined in (\ref{eqn:dpalgo-marg}). The correctness of these outputs is shown in Theorem~\ref{thm:dpalgo}\ref{sthm:dpalgo2}. 
\item $\big(\overline{\zeta}_{t}^{(m)}\big)_{t=1:T,m=1:M}$ are the expected values of the quantities of interest $\big(\zeta^{(m)}_t(\cdot,\cdot,\cdot)\big)_{m=1:M}$ defined in (\ref{eqn:qoi-expval}). $\big(\overline{Z}_{\zeta^{(m)}}\big)_{m=1:M}$ are the aggregate expected values of these quantities of interest defined in (\ref{eqn:qoi-aggregate}). The correctness of these outputs is shown in Theorem~\ref{thm:dpalgo}\ref{sthm:dpalgo3}. 
\end{itemize}
Below is a list explaining some of the lines in Algorithm~\ref{alg:dpconcrete}. 
\begin{itemize}
\item Line~\ref{alglin:dp-b-b-def} computes the lowest possible Bonus-Malus level $\underline{b}$ and the highest possible Bonus-Malus level $\overline{b}$ that the insured can transition to from level $b$. We only iterate through the Bonus-Malus levels between $\underline{b}$ and $\overline{b}$ in the remainder of the for-loop starting from Line~\ref{alglin:dp-forloop-b}.
\item Line~\ref{alglin:dp-L-def} computes a set $\CL_t(b,b')\subseteq\R_+$ which contains all potential insurance compensation amounts (i.e., $\lambda^{\CB\CM}(b,t,L(d,W_t))$) such that the optimal decision of the insured is to make a claim, and that the updated Bonus-Malus level after making the claim is $b'$. 
In particular, the following claim holds true.
\begin{enumerate}[label=\underline{Claim~\arabic*}:,leftmargin=3.8em]
\item for all $(b,i)\in\CB\times\CI$, $d\in\CD$, $w\in\CW$, the following equivalence holds:
\begin{align*}
\begin{split}
g_{t}(b,i,d,1,1,w)+\CV_{t}\big(f_{t}(b,i,d,1,1,w)\big)&< g_{t}(b,i,d,1,0,w)+\CV_{t}\big(f_{t}(b,i,d,1,0,w)\big)\\
&\Updownarrow\\
\lambda^{\CB\CM}(b,t,L(d,w))&\in\bigcup_{\underline{b}\le b'\le\overline{b}}\CL_t(b,b').
\end{split}
\end{align*}
\end{enumerate}
This claim is justified in (\ref{eqn:dpalgo-claim1-proof1}) and (\ref{eqn:dpalgo-claim1-proof2}) in the proof of Theorem~\ref{thm:dpalgo}\ref{sthm:dpalgo1}. 
Moreover, notice that $\CL_t(b,b')$ is an interval since the set $\left\{c\in\R_+:\CB\CM(b,c)=b'\right\}$ is an interval in $\R_+$ by the assumption that the function $\CB\CM(b,\cdot)$ is non-decreasing.
\item In the for-loop starting from Line~\ref{alglin:dp-forloop-d}, we compute temporary values $(H_t(b,i,d,\iota))_{d\in\CD,\iota\in\{0,1\}}$ which are subsequently used in Line~\ref{alglin:dp-d-iota-def} and Line~\ref{alglin:dp-value-update}. In fact, the following claim holds true.
\begin{enumerate}[label=\underline{Claim~\arabic*}:,leftmargin=3.8em]
\setcounter{enumi}{1}
\item for all $(b,i)\in\CB\times\CI$, $d\in\CD$, $\iota\in\{0,1\}$, it holds that
\begin{align*}
\begin{split}
&\phantom{=}\;\;\EXP\bigg[\min_{j\in\{0,1\}}\Big\{g_{t}(b,i,d,\iota,j,W_t)+\CV_{t}\big(f_{t}(b,i,d,\iota,j,W_t)\big)\Big\}\bigg]\\
&=\beta(d)+\iota p^{\CB\CM}(b,t)+\delta_{\Ti\Tn}(t)\INDI_{\{i=\Tn\To,\iota=1\}}+\delta_{\To\Tu\Tt}(t)\INDI_{\{i=\To\Tn,\iota=0\}}\\
&\qquad+\delta_{\Tr\Te}\INDI_{\{i\ne\To\Tn,i\ne\Tn\To,\iota=1\}}+\EXP\big[L(d,W_t)\big]+H_t(b,i,d,\iota).
\end{split}
\end{align*}
\end{enumerate}
This claim is justified in (\ref{eqn:dpalgo-claim2-proof1}) and (\ref{eqn:dpalgo-claim2-proof2}) in the proof of Theorem~\ref{thm:dpalgo}\ref{sthm:dpalgo1}. 
\item The main step of the dynamic programming algorithm are in Line~\ref{alglin:dp-value-init}, Line~\ref{alglin:dp-d-iota-def}, Line~\ref{alglin:dp-j-def}, and Line~\ref{alglin:dp-value-update}. In Line~\ref{alglin:dp-value-init}, the values of $(\CV_T(b,i))_{(b,i)\in\CB\times\CI}$ are all initialised to 0 according to (\ref{eqn:dpvalue}). 
In Line~\ref{alglin:dp-d-iota-def} and Line~\ref{alglin:dp-j-def}, the one-stage optimisers $(\widehat{d}_t,\widehat{\iota}_t,\widehat{j}_t)$ are computed as in (\ref{eqn:dpconcept-d-iota}) and (\ref{eqn:dpconcept-j}), due to \underline{Claim~1} and \underline{Claim~2} above. 
In Line~\ref{alglin:dp-value-update}, the value of $\CV_t(b,i)$ is computed as in (\ref{eqn:dpvalue}), due to \underline{Claim~2} above.
\end{itemize}
\label{rmk:dpalgo}
\end{remark}

\begin{theorem}
Let us assume that all inputs of Algorithm~\ref{alg:dpconcrete} are set according to Remark~\ref{rmk:dpalgo}, and let $\CV_0(0,\Tn\To)$, $\pi^{\star}$, $\big(P^\star_t\big)_{t=1:T}$, $\big(\overline{P}^\star_t\big)_{t=0:T}$, $\big(\overline{\zeta}_{t}^{(m)}\big)_{t=1:T,m=1:M}$, and $\big(\overline{Z}_{\zeta^{(m)}}\big)_{m=1:M}$ be the outputs of Algorithm~\ref{alg:dpconcrete}. 
Then, the following statements hold.
\begin{enumerate}[label=(\roman*)]
\item \label{sthm:dpalgo1}$\CV_0(0,\Tn\To)=\BBV_0$ and $\pi^\star$ is an optimal decision policy for (\ref{eqn:stocoptctr}), i.e., $\CV_0(0,\Tn\To)=V^{\pi^\star}_0=\BBV_0$. 

\item \label{sthm:dpalgo2}$\big(P^\star_t\big)_{t=1:T}$ are equal to the transition kernels of $\big(b^{\pi^\star}_t,i^{\pi^\star}_t\big)_{t=0:T}$ defined in (\ref{eqn:dpalgo-kernel}), and $\big(\overline{P}^\star_t\big)_{t=0:T}$ are equal to the marginal state occupancy probabilities of $\big(b^{\pi^\star}_t,i^{\pi^\star}_t\big)_{t=0:T}$ defined in (\ref{eqn:dpalgo-marg}).

\item \label{sthm:dpalgo3}$\big(\overline{\zeta}_{t}^{(m)}\big)_{t=1:T,m=1:M}$ are equal to the expected values of the quantities of interest $\big(\zeta^{(m)}_t(\cdot,\cdot,\cdot)\big)_{m=1:M}$ defined in (\ref{eqn:qoi-expval}), and $\big(\overline{Z}_{\zeta^{(m)}}\big)_{m=1:M}$ are equal to the aggregate expected values of these quantities of interest defined in (\ref{eqn:qoi-aggregate}).

\end{enumerate}
\label{thm:dpalgo}
\end{theorem}

\begin{proof}
See Appendix~\ref{apx:proofs}. 
\end{proof}

\begin{remark}
Assume that the following quantities either admit an analytically tractable expression, or can be efficiently approximated to high numerical precision:
\begin{enumerate}[label=(\roman*)]
\item the expectation $\EXP\big[L(d,W_t)\big]=\EXP[N_t]\EXP\big[(X^{(t)}-\gamma(d))^+\big]$, where $d\in\CD$ and $t\in\{1,\ldots,T\}$; \label{srmk:comp1}
\item the expectation $\EXP\!\left[\INDI_{I}(\lambda^{\CB\CM}(b,t,L(d,W_t)))\Big(\lambda^{\CB\CM}(b,t,L(d,W_t))-\alpha\Big)^+\right]$, where $b\in\CB$, $d\in\CD$, $t\in\{1,\ldots,T\}$, $\alpha\ge0$, and $I\subset\R_+$ is an interval; \label{srmk:comp2}
\item the probability $\PROB\left[\lambda^{\CB\CM}(b,t,L(d,W_t))\in I\right]$, where $b\in\CB$, $d\in\CD$, $t\in\{1,\ldots,T\}$, and $I\subset\R_+$ is an interval; \label{srmk:comp3}
\item the expectation $\EXP\big[\zeta^{(m)}_t(b,i,W_t)\big]$, where $b\in\CB$, $i\in\CI$, $m\in\{1,\ldots,M\}$, $t\in\{1,\ldots,T\}$. \label{srmk:comp4}
\end{enumerate}
Then, Algorithm~\ref{alg:dpconcrete} is computationally tractable, meaning that quantities in Algorithm~\ref{alg:dpconcrete} can either be computed exactly or efficiently approximated to high numerical precision. 
In addition, since $\lambda^{\CB\CM}(b,t,L(d,W_t))$ is bounded above by $l^{\CB\CM}_{\max}(b,t)$, we can assume without loss of generality that the interval $I$ in \ref{srmk:comp2} and \ref{srmk:comp3} above is bounded. A concrete model in which Algorithm~\ref{alg:dpconcrete} is computationally tractable will be introduced in Section~\ref{sec:g-and-h}. 
\label{rmk:comp}
\end{remark}

\subsection{Discussion About the Pricing of the Insurance Premium}
\label{ssec:pricing}
One important consideration of the cyber risk insurer is the choice of the annual premium. Conventionally, for a fixed self-mitigation measure $d\in\CD$, a fixed deductible $l_{\Td\Tt\Tb}\ge 0$, and a fixed maximum compensation $l_{\max}\ge0$, one may consider the risk premium $\EXP[(L(d,W_t)-l_{\Td\Tt\Tb})^+\wedge l_{\max}]$. 
However, these quantities are variable in our dynamic model with the Bonus-Malus system.
Moreover, the insured may choose to withdraw from the cyber risk insurance contract, thus further complicating the matter.

Even though the optimal cybersecurity provisioning process minimises the total cybersecurity cost from the insured's point of view, the insurer can use the quantities of interest computed from this process to evaluate the insurance contract and compare different contract design options.
One of the evaluation criteria is how high the insurer can price the premium while retaining their customers, that is, what is the highest premium such that the cyber risk insurance policy remains attractive to a rational insured. 
To answer this question, one can analyse the relationship between the marginal state occupancy probabilities $\big(\overline{P}^\star_t\big)_{t=0:T}$ and the premium.
When the premium is high, the insured may opt not to purchase the cyber risk insurance policy, or may withdraw from the insurance contract prematurely due to a transition into a higher Bonus-Malus level. 
This evaluation criterion is crucial for the insurer, since a high retention rate would allow the insurer to build a large pool of customers with similar risk exposures.

Another criterion is the range of premium under which a rational insured would choose to adopt a self-mitigation measure. 
The reasoning here is that if the insurance premium is set too low, then the insured will see cyber risk insurance as a low-cost alternative to self-mitigation measures and will thus choose not to adopt any self-mitigation measure.  
This range can be determined by studying the relationship between the values of
$\big(\overline{\zeta}^{(d+1)}_t\big)_{t=1:T,d\in\CD}$ in Example~\ref{exp:quantities}\ref{sexp:quantities-mitiadoption} and the premium.
Moreover, the value of $\overline{Z}_{\zeta^{(D+4)}}$ in Example~\ref{exp:quantities}\ref{sexp:quantities-mitieffect} reflects the discounted total amount of loss prevented by the adoption of self-mitigation measures, and its relationship with the premium can also be studied.
These are important criteria due to the positive externalities generated by organisations adopting self-mitigation measures against cyber threats.
Combining the criterion of loss prevention with the criterion of insurance customer retention, one can study whether the adoption of cyber risk insurance disincentivises the adoption of self-mitigation measures, which is an issue known as moral hazard.

An important question related to the design of the insurance contract is how much the insurer can financially gain from offering the cyber risk insurance policy. One simple criterion is to look at the discounted expectation of the difference between the total payment from the insured to the insurer and the total insurance compensation. Specifically, let $\overline{Z}_{\Ti\Tn\Ts}\equiv\overline{Z}_{\zeta^{(D+3)}}$ in Example~\ref{exp:quantities}\ref{sexp:quantities-premium} and let $\overline{Z}_{\Tc\Tp}\equiv\overline{Z}_{\zeta^{(D+5)}}$ in Example~\ref{exp:quantities}\ref{sexp:quantities-benefit}. 
Then, one can see $\overline{Z}_{\Ti\Tn\Ts}-\overline{Z}_{\Tc\Tp}$ as the insurer's expected profit from the cyber risk insurance policy.
However, this approach implicitly makes the zero-sum assumption, i.e., the insurer's expected profit is the negative of the expected net payment from the insured to the insurer. 
Since we assume that the insured always has the option to not purchase cyber risk insurance, which would result in zero net payment to the insurer, the quantity $\overline{Z}_{\Ti\Tn\Ts}-\overline{Z}_{\Tc\Tp}$ will always be non-positive. 
In particular, in the case where the insurance policy is not purchased, the quantity $\overline{Z}_{\Ti\Tn\Ts}-\overline{Z}_{\Tc\Tp}$ will be exactly 0.
In reality, there are a few important practical considerations that this zero-sum approach neglects.
First, the insured and the insurer often perceive cyber risk differently. 
Specifically, this is due to the insured having impartial or incomplete information as to the frequency and severity of the risk they face relative to an insurer who is privy to a portfolio of insured, giving them a competitive advantage in their assessments of such loss processes.
Second, the zero-sum approach assumes that the insured will act optimally according to the optimal cybersecurity provisioning process.
In reality, the insured will typically act sub-optimally due to the practical complexity of the required computation being beyond the scope and the available resources of the decision makers \citep{targino2013optimal}.
Third, regulatory transparency requirements will also play an important role in determining the profit margins that may arise if such products are issued.
The lack of these considerations indicates some of the limitations of this zero-sum approach. 
Nevertheless, the relationship between the quantity $\overline{Z}_{\Ti\Tn\Ts}-\overline{Z}_{\Tc\Tp}$ and the premium does shed some light on an important aspect of cyber risk insurance contract, that is, whether it is possible to make this quantity arbitrarily close to 0 by setting the premium appropriately. 
If an insurance contract design allows such possibility, then there is a way to set the premium such that the insured becomes indifferent between purchasing and not purchasing the insurance policy up to an arbitrarily small tolerance, at which point the value of $\overline{Z}_{\Ti\Tn\Ts}-\overline{Z}_{\Tc\Tp}$ would be approximately maximised. 
Then, the real profit of the insurer should become more preferable than in the case where the value of $\overline{Z}_{\Ti\Tn\Ts}-\overline{Z}_{\Tc\Tp}$ is far from zero.
On the contrary, if a design does not allow such possibility, then there is no way to set the premium such that the insured becomes approximately indifferent between purchasing and not purchasing the insurance policy, indicating that there is space for improving the contract design in order to hopefully further improve the real profit of the insurer.
We will discuss the aforementioned important criteria when evaluating the design of cyber risk insurance contracts with a concrete example in Section~\ref{sec:exp}.

\section{Modelling Cyber Loss with Truncated g-and-h Distribution}
\label{sec:g-and-h}

In this section, we adopt specific distributional assumptions about the random variables $\big(X^{(t)}\big)_{t=1:T}$, where $X^{(t)}$ corresponds to the severity of a single cyber loss event in year $t$. 
Studies such as \citep{maillart2010heavy, wheatley2016extreme} and \citep{dacorogna2023building} have shown that the severity of cyber loss events have left-skewed and heavy-tailed distributions. 
A commonly adopted family of distributions in statistical analyses of cyber event data is the Generalised Pareto distribution; see, e.g., \citep{farkas2021cyber, malavasi2022cyber, dacorogna2023building}. 
Generalised Pareto distributions naturally appear when analysing heavy-tailed random variables as a consequence of the Pickands--Balkema--De Haan theorem from extreme value theory; see, e.g., \citep{embrechts1997modelling}.
Since the Generalised Pareto distribution arises as an approximation of the tail part of a distribution, it is also common to model the body and the tail parts of a distribution separately; see, e.g., \citep{zangerle2023modelling} which models cyber losses below a certain threshold by a normal distribution and models cyber losses exceeding the threshold by a Generalised Pareto distribution.
It is well-known that the g-and-h family introduced by \citet{tukey1977exploratory} contains distributions with a wide range of skewness and kurtosis (see, e.g., Figure~3 of \citep{dutta2006tale}), which makes it suitable for modelling Operational Risk \citep{dutta2006tale,peters2006bayesian,cruz2015fundamental,peters2015advances}. 
Using the results from empirical analyses of a real Operational Risk dataset, \citet{dutta2006tale} identified the following advantages of the g-and-h distribution compared to other more commonly adopted alternatives (including the Generalised Pareto distribution): 
\begin{itemize}
\item it is flexible and it fits well to real data under many different circumstances, e.g., when considering losses of a company as a whole and when considering individual business lines or event types; and
\item it produces realistic estimations of the Operational Risk capital.
\end{itemize}
Moreover, it is straightforward and efficient to simulate random samples from the g-and-h distribution.
The parameters in the g-and-h distribution can be robustly estimated based on quantiles~\citep{xu2014robust} or L-moments~\citep{peters2016estimating}. 
Due to these properties, we adopt the g-and-h distribution as a particular model for the severity of cyber loss events in this section. 
We would like to remark that the aforementioned studies using the Generalised Pareto distribution consider severity distributions with parameters that depend on a large collection of covariates, including those related to the type, size, and location of the organisation, as well as the type of cyber loss event. 
Since we are interested in analysing the optimal cybersecurity provisioning process from the perspective of a single organisation facing a single type of cyber loss event, we assume that the parameters in the severity distribution have been adjusted to the characteristics of the organisation and the type of cyber loss event of interest. 
In practice, calibration of the cyber loss model can combine information from historical data from other organisations, internal data, and expert opinions \citep{malavasi2022cyber}.

The g-and-h distribution is a four-parameter family of distributions, given by the following definition:
\begin{align}
\begin{split}
\widetilde{X}\text{ follows a g-and-h}&(\mu,\varsigma,g,h)\text{ distribution, if}\\
\widetilde{X}&=\mu+\varsigma Y_{g,h}(Z),\\
\text{where }Z&\sim\text{Normal}(0,1),\\
Y_{g,h}(z)&:=\begin{cases}
\frac{\exp(gz)-1}{g}\exp\left(\frac{hz^2}{2}\right) & \text{if }g\ne 0,\\
z\exp\left(\frac{hz^2}{2}\right) & \text{if }g=0,
\end{cases}
\end{split}
\label{eqn:g-and-hdef}
\end{align}
where $\mu\in\R$ is the location parameter, $\varsigma>0$ is the scale parameter, $g\in\R$ is the skewness parameter, and $h\ge0$ is the kurtosis parameter. 
In this paper, we assume that the parameters $\mu$, $\varsigma$, $g$, and $h$ are fixed and known.  
By (\ref{eqn:g-and-hdef}), the distribution function of $\widetilde{X}$ is given by
\begin{align}
\begin{split}
F_{\widetilde{X}}(x):=\PROB[\widetilde{X}\le x]=\Phi\left(Y_{g,h}^{-1}\left(\tfrac{x-\mu}{\varsigma}\right)\right),
\end{split}
\end{align}
where $Y_{g,h}^{-1}$ denotes the inverse function of $Y_{g,h}$, and $\Phi$ denotes the distribution function of the standard normal distribution. 
Even though $Y_{g,h}^{-1}$ cannot be expressed analytically, it can be efficiently evaluated using a standard root-finding procedure such as the bisection method and Newton's method. Therefore, we treat $Y_{g,h}^{-1}$ as a tractable function. 
The g-and-h distribution has the property that the $m$-th moment of $\widetilde{X}$ exists when $h<\frac{1}{m}$ (see, e.g., Appendix~D of \citep{dutta2006tale}). Since we consider losses that are positively skewed and have finite expectation (see Remark~\ref{rmk:finite-expectation}), from now on, we assume that $g>0$ and $0\le h<1$. 

Since cyber losses are positive, we introduce a truncated version of the g-and-h distribution. 
\begin{definition}[Truncated g-and-h distribution] For $\mu\in\R,\varsigma>0,g>0,h\in[0,1)$, the random variable $X$ has truncated g-and-h distribution with parameters $\mu,\varsigma,g,h$, denoted by $X\sim\text{Tr-g-and-h}(\mu,\varsigma,g,h)$, if $X$ has distribution function 
\begin{align}
F_X(x):=\PROB[X\le x]=\PROB[\widetilde{X}\le x|\widetilde{X}>0],
\label{eqn:tr-g-and-hdef}
\end{align}
where $\widetilde{X}\sim\text{g-and-h}(\mu,\varsigma,g,h)$. 
\label{def:tr-g-and-h}
\end{definition}

The next lemma shows some useful properties of the truncated g-and-h distribution.
\begin{lemma}
Suppose that $X\sim\text{Tr-g-and-h}(\mu,\varsigma,g,h)$ for $\mu\in\R,\varsigma>0,g>0,h\in[0,1)$. Then, the following statements hold.
\begin{enumerate}[label=(\roman*)]
\item The distribution function of $X$ is given by
\begin{align}
\begin{split}
F_{X}(x)=\begin{cases}
\frac{F_{\widetilde{X}}(x)-F_{\widetilde{X}}(0)}{1-F_{\widetilde{X}}(0)} & \text{if }x>0,\\
0 & \text{if }x\le0,
\end{cases}
\end{split}
\label{eqn:tr-g-and-h-df}
\end{align}
where $F_{\widetilde{X}}$ is defined in (\ref{eqn:g-and-hdef}). 
\label{slem:tr-g-and-h1}
\item Suppose that $U\sim\text{Uniform}[0,1]$, and let
\begin{align}
X_U:=\mu+\varsigma Y_{g,h}\Big(\Phi^{-1}\Big(U+(1-U)F_{\widetilde{X}}(0)\Big)\Big),
\label{eqn:tr-g-and-h-inv}
\end{align}
then $X_U\sim\text{Tr-g-and-h}(\mu,\varsigma,g,h)$. 
\label{slem:tr-g-and-h2}
\item For any $\gamma\ge0$, the expectation $\EXP\big[(X-\gamma)^+\big]$ is given by:
\begin{align}
\begin{split}
\EXP\big[(X-\gamma)^+\big]&=\frac{\varsigma}{(1-F_{\widetilde{X}}(0))g\sqrt{1-h}}\Bigg[\exp\left(\frac{g^2}{2(1-h)}\right)\Phi\left(\left(\frac{g}{1-h}-Y_{g,h}^{-1}\left(\tfrac{\gamma-\mu}{\varsigma}\right)\right)\sqrt{1-h}\right)\\
&\qquad-\Phi\left(-Y_{g,h}^{-1}\left(\tfrac{\gamma-\mu}{\varsigma}\right)\sqrt{1-h}\right)\Bigg]+\frac{(\mu-\gamma)(1-F_{\widetilde{X}}(\gamma))}{1-F_{\widetilde{X}}(0)}.
\end{split}
\label{eqn:tr-g-and-h-mom}
\end{align}
\label{slem:tr-g-and-h3}
\end{enumerate}
\label{lem:tr-g-and-h}
\end{lemma}
\begin{proof}
See Appendix~\ref{apx:proofs}. 
\end{proof}

Suppose that $X^{(t)}$ follows a truncated g-and-h distribution for $t=1,\ldots,T$,
Lemma~\ref{lem:tr-g-and-h}\ref{slem:tr-g-and-h2} allows us to efficiently generate random samples from the severity distributions $F_{X^{(1)}},\ldots,F_{X^{(T)}}$, thus allowing us to approximate the distributions of quantities of interest in Example~\ref{exp:quantities} via Monte Carlo.
Lemma~\ref{lem:tr-g-and-h}\ref{slem:tr-g-and-h3} shows that the assumption~\ref{srmk:comp1} in Remark~\ref{rmk:comp} is satisfied as long as for $t=1,\ldots,T$, the expected value of the frequency distribution in year $t$, i.e., $\EXP[N_t]$, is also tractable. 
Lemma~\ref{lem:tr-g-and-h}\ref{slem:tr-g-and-h1} provides the distribution function of $X^{(t)}$ that can be used to approximate the distribution function of $L(d,W_t)$, for $d\in\CD$ and $t=1,\ldots,T$. Concretely, by adopting the fast Fourier transform (FFT) approach with exponential tilting (see, e.g., \citep{embrechts2009panjer, cruz2015fundamental}), we approximate the distribution function of $L(d,W_t)$, denoted by $F_{L(d,W_t)}$, by a finitely supported discrete distribution $\widehat{F}_{L(d,W_t)}(x)=\sum_{j\in\CA}p^{(t,d)}_j\INDI_{\big\{a^{(t,d)}_j\le x\big\}}$, where $\CA\subset\Z_+$ is a finite index set, $\big(a^{(t,d)}_j\big)_{j\in\CA}\subset\R_+$ is a finite set of atoms and $\big(p^{(t,d)}_j\big)_{j\in\CA}$ are the corresponding probabilities. 
The details of the FFT approach with exponential tilting are shown in Algorithm~\ref{alg:fft}. 
After obtaining $(\widehat{F}_{L(d,W_t)})_{t=1:T,\,d\in\CD}$ from Algorithm~\ref{alg:fft}, the quantities $\EXP\!\left[\INDI_{I}(\lambda^{\CB\CM}(b,t,L(d,W_t)))\big(\lambda^{\CB\CM}(b,t,L(d,W_t))-\alpha\big)^+\!\right]$ and $\PROB\left[\lambda^{\CB\CM}(b,t,L(d,W_t))\in I\right]$ in Remark~\ref{rmk:comp} can be approximated by finite sums:
\begin{align*}
&\EXP\!\left[\INDI_{I}(\lambda^{\CB\CM}(b,t,L(d,W_t)))\big(\lambda^{\CB\CM}(b,t,L(d,W_t))-\alpha\big)^+\!\right]\\
&\qquad\qquad\qquad\qquad\approx\sum_{j\in\CA}p^{(t,d)}_j\!\INDI_{I}\big(\lambda^{\CB\CM}\big(b,t,a^{(t,d)}_j\big)\big)\big(\lambda^{\CB\CM}\big(b,t,a^{(t,d)}_j\big)-\alpha\big)^+,\\
&\PROB\!\left[\lambda^{\CB\CM}(b,t,L(d,W_t))\in I\right]\\
&\qquad\qquad\qquad\qquad\approx\sum_{j\in\CA}p^{(t,d)}_j\!\INDI_{I}\big(\lambda^{\CB\CM}\big(b,t,a^{(t,d)}_j\big)\big).
\end{align*}
One may increase the granularity parameter $K_{\Tg\Tr}$ in Algorithm~\ref{alg:fft} to increase the precision of numerical approximation. 
Consequently, assumptions~\ref{srmk:comp2} and \ref{srmk:comp3} in Remark~\ref{rmk:comp} are satisfied, and hence, Algorithm~\ref{alg:dpconcrete} is tractable and efficient in this setting. In particular, Algorithm~\ref{alg:fft} only needs to be executed once before executing Algorithm~\ref{alg:dpconcrete}. 

\begin{algorithm}[t]
\KwIn{$T$, $\CD$, \big($F_{X^{(t)}}(\cdot)\big)_{t=1:T}$, $\big(\psi_{N_t}(\cdot)\big)_{t=1:T}$, $\gamma(\cdot)$, $\overline{l}$, $K_{\Tg\Tr}\in\N$, $\theta>0$}
\KwOut{$(a^{(t,d)}_j,p_j^{(t,d)})_{j\in\CA,\,t=1:T,\,d\in\CD}$, $\widehat{F}_{L(d,W_t)}(x)=\sum_{j\in\CA}p^{(t,d)}_j\INDI_{\{a^{(t,d)}_j\le x\}}$ for $t=1,\ldots,T$ and each $d\in\CD$}
\nl $\varepsilon\leftarrow(2^{K_{\Tg\Tr}}-1)^{-1}\overline{l}$, $\CA\leftarrow\{0,1,\ldots,2^{K_{\Tg\Tr}}-1\}$. \\
\nl \For{$t=1,\ldots,T$, $d\in\CD$}{
\nl $a^{(t,d)}_j\leftarrow j\epsilon$ for each $j\in\CA$. \\
\nl $f^{(t,d)}_{j}\leftarrow \exp(-j\theta)\left[F_{X^{(t)},d}(j\varepsilon+\frac{1}{2}\varepsilon)-F_{X^{(t)},d}(j\varepsilon-\frac{1}{2}\varepsilon)\right]$ for each $j\in\CA$, where $F_{X^{(t)},d}(y):=F_{X^{(t)}}\big(y+\gamma(d)\big)\INDI_{\{y\ge 0\}}$. \\
\nl $\varphi^{(t,d)}_j\leftarrow\sum_{k\in\CA}\exp(i\pi2^{1-K_{\Tg\Tr}} jk)f^{(t,d)}_k$ for each $j\in\CA$ via the FFT algorithm. \\
\nl $\psi^{(t,d)}_j\leftarrow\psi_{N_t}(\varphi^{(t,d)}_j)$ for each $j\in\CA$. \\
\nl $p^{(t,d)}_j\leftarrow\exp(j\theta)2^{-K_{\Tg\Tr}}\sum_{k\in\CA}\exp(-i\pi2^{1-K_{\Tg\Tr}} jk)\psi^{(t,d)}_k$ for each $j\in\CA$ via the inverse FFT algorithm. \\
}
\nl \Return $(a^{(t,d)}_j,p_j^{(t,d)})_{j\in\CA,\,t=1:T,\,d\in\CD}$, $\widehat{F}_{L(d,W_t)}(x)=\sum_{j\in\CA}p^{(t,d)}_j\INDI_{\{a^{(t,d)}_j\le x\}}$ for $t=1,\ldots,T$ and each $d\in\CD$.\\
    \caption{{\bf Fast Fourier Transform Approach with Exponential Tilting for Approximating $F_{L(d,W_t)}$} (see \citep{embrechts2009panjer})}
    \label{alg:fft}
\end{algorithm}

\section{Numerical Experiment}
\label{sec:exp}

In Section~\ref{sec:stocoptctr} and Section~\ref{sec:g-and-h}, we formulated the optimal cybersecurity provisioning problem as a finite horizon stochastic optimal control problem, and developed a dynamic programming algorithm, i.e., Algorithm~\ref{alg:dpconcrete}, to efficiently solve the problem under the assumption of truncated g-and-h loss severity distributions. Algorithm~\ref{alg:dpconcrete} not only computes the optimal cybersecurity provisioning policy $\pi^\star$ for the insured, but also computes related quantities of interest, including the marginal state occupancy probabilities $\big(\overline{P}^\star_t\big)_{t=0:T}$ of the optimally controlled process $\big(b^{\pi^\star}_t,i^{\pi^\star}_t\big)_{t=0:T}$ as well as other quantities of interest such as those illustrated in Example~\ref{exp:quantities}.
As discussed in Section~\ref{ssec:pricing}, these quantities can guide the insurer when designing a suitable cyber risk insurance contract with a Bonus-Malus system. 
In this section, we demonstrate how Algorithm~\ref{alg:dpconcrete} aids the insurer when designing a cyber risk insurance contract and the benefits of the Bonus-Malus system by a numerical experiment.\footnote{The code used in this work for the experiment is available on GitHub: \url{https://github.com/qikunxiang/CyberInsuranceBonusMalus}} 
In particular, we investigate two aspects of the cyber risk insurance contract with Bonus-Malus discussed in Section~\ref{ssec:pricing}. 
The first aspect is the issue of moral hazard, that is, whether the presence of the cyber risk insurance contract disincentivises the adoption of self-mitigation measures. 
The second aspect is whether the Bonus-Malus system provides benefits to the insurer in terms of increased customer retention rates and expected profits.

\subsection{Experimental Settings}

We assume that all monetary quantities, including the severity of cyber loss events, the insurance premium, and the annual investment required by the self-mitigation measures are adjusted to the scale of the insured (e.g., their average annual revenue) and are unit-free. 
We consider insurance policies that last for 20 years, that is, $T=20$. The discount factor $e^{-r}$ is fixed at 0.95. 

Next, let us specify the frequency distributions and the severity distributions in the cyber loss model.
Existing studies have constructed and calibrated sophisticated frequency and severity models which involve a large number of covariates related to the type, size, and location of an organisation, as well as its internal IT security measures and technologies.
Examples of these models include regression trees \citep{farkas2021cyber}, causal models \citep{skarczinski2022more}, and the Generalised Additive Models for Location Shape and Scale (GAMLSS) \citep{malavasi2022cyber}. 
We would like to remark that calibrating model parameters is challenging in practice due to the scarcity of reliable datasets for cyber losses, as we have discussed in Section~\ref{ssec:challenges}. 
Moreover, calibrating models solely based on historical data is backward looking, and an ideal calibration process should involve historical data, internal data, and expert opinions (in the form of scenario analyses); see \citep[Remark~1]{malavasi2022cyber}.
Nevertheless, since the numerical experiment we are conducting aims to illustrate the workings of our optimal cybersecurity provisioning process and especially to showcase the benefits of the Bonus-Malus system, we consider for simplicity a generic organisation and opt to use static frequency and severity models without any covariates.
We set the frequency distribution in every policy year to be the Poisson distribution with rate 0.8, that is, $N_t\sim\mathrm{Poisson}(0.8)$ for $t=1,\ldots,T$, and we set the severity distribution in every policy year to be $X^{(t)}\sim\text{Tr-}g\text{-and-}h(\mu=0,\varsigma=1,g=1.8,h=0.15)$ for $t=1,\ldots,T$, where the $g$ and $h$ parameters are set to be similar to those estimated in \citep{dutta2006tale} from real Operational Risk data (see Table~8 of \citep{dutta2006tale}). 
We would like to remark that the heaviness of the tail of the loss severity distribution (i.e., the parameter $h$ in the truncated $g$-and-$h$ distribution) determines the probability of extreme risk events and is crucial in the computation of capital estimates \citep{malavasi2022cyber}. 
Therefore, setting the value of the parameter $h$ to be similar to the value estimated by \citet{dutta2006tale} ensures that it is within the range of realistic values for real organisations facing cyber threats.
Moreover, since our cyber loss model is flexible and the code used in the experiment is publicly available on GitHub, future studies can directly use our model and algorithms with fine-tuned frequency and severity distributions that are calibrated to the characteristics of real organisations.

For simplicity, we consider the situation where only a single self-mitigation measure is available, that is, $D=1$. This self-mitigation measure requires an annual investment of 0.5, and has the effect of preventing 70\% of the incoming cyber loss events and decreasing the severity of the remaining events by the 70th percentile of the severity distribution, that is, $\beta(1)=0.5$, $\gamma(1)=F^{-1}_{X}(0.7)$, where $X\sim\text{Tr-}g\text{-and-}h(\mu=0,\varsigma=1,g=1.8,h=0.15)$. 
We consider the following simple cyber risk insurance policy with Bonus-Malus system. Let $\CB=\{-2,-1,0,1\}$, and let the functions $\CB\CM(b_{t-1},C_t)$ and $\CB\CM_0(b_{t-1},i_{t-1})$ be specified in Table~\ref{tab:bm} below.

\begin{table}[h]
\label{tab:bm}
\caption{The $\CB\CM(\cdot,\cdot)$ and $\CB\CM_0(\cdot,\cdot)$ functions that represent the Bonus-Malus update rules}
\begin{center}
\begin{tabular}{|c|c|c|c|}
\hline 
\multicolumn{2}{|c|}{\multirow{2}{*}{$\CB\CM(b_{t-1},C_t)$}} & \multicolumn{2}{|c|}{$C_t$} \\ 
\cline{3-4}
\multicolumn{2}{|c|}{} & $=0$ & $>0$ \\ 
\cline{1-4}
\multirow{4}{*}{$b_{t-1}$} & $-2$ & $-2$ & $1$ \\ 
\cline{2-4}
 & $-1$ & $-2$ & $1$ \\ 
\cline{2-4}
 & $0$ & $-1$ & $1$ \\ 
\cline{2-4}
 & $1$ & $0$ & $1$ \\ 
\hline 
\end{tabular}
\hspace{1cm}
\begin{tabular}{|c|c|c|c|}
\hline 
\multicolumn{2}{|c|}{\multirow{2}{*}{$\CB\CM_0(b_{t-1},i_{t-1})$}} & \multicolumn{2}{|c|}{$i_{t-1}$} \\ 
\cline{3-4}
\multicolumn{2}{|c|}{} & $\text{on}$ & $\text{off}_1$ \\ 
\cline{1-4}
\multirow{4}{*}{$b_{t-1}$} & $-2$ & $(-2,\text{off}_1)$ & $(-1,\text{off}_1)$ \\ 
\cline{2-4}
 & $-1$ & $(-1,\text{off}_1)$ & $(0,\text{off}_1)$ \\ 
\cline{2-4}
 & $0$ & $(0,\text{off}_1)$ & $(0,\text{off}_1)$ \\ 
\cline{2-4}
 & $1$ & $(1,\text{off}_1)$ & $(0,\text{off}_1)$ \\ 
\hline 
\end{tabular}
\end{center}
\end{table}
The above settings mean that when the contract is activated, the insured is migrated to level 1 in the subsequent policy year whenever a claim is made. When the insured does not make any claim in a policy year, their policy is migrated downwards by one level in the subsequent policy years until it reaches level $-2$. When the contract is deactivated, if the insured's policy is in level 1, it is migrated back to level 0 after one year. Otherwise, the policy is migrated upwards by one level each year until it reaches level 0. 
In the experiment, we let the base premium $p^{\CB\CM}_{\text{base}}$ be an adjustable parameter that is varied between 0 and 7 with an increment of $0.005$, and set the premium to be $60\%,80\%,100\%,150\%$ of the base premium for the Bonus-Malus levels $-2,-1,0,1$, respectively. That is, we let $p^{\CB\CM}(-2,t)=0.6p^{\CB\CM}_{\text{base}}$, $p^{\CB\CM}(-1,t)=0.8p^{\CB\CM}_{\text{base}}$, $p^{\CB\CM}(0,t)=p^{\CB\CM}_{\text{base}}$, $p^{\CB\CM}(1,t)=1.5p^{\CB\CM}_{\text{base}}$ for all $t\in\{1,\ldots,T\}$. We fix the maximum compensation to be 1000, that is, $l^{\CB\CM}_{\max}(b,t)=1000$ for all $b\in\CB,t\in\{1,\ldots,T\}$.
We set the deductible to be 0.5 for all but the final policy year, and set the deductible to be 5 for the final policy year, that is, $l^{\CB\CM}_{\Td\Tt\Tb}(b,t)=0.5$ for all $b\in\CB,t\in\{1,\ldots,T-1\}$ and $l^{\CB\CM}_{\Td\Tt\Tb}(b,T)=5$ for all $b\in\CB$. This is to prevent an issue caused by the finite horizon. Since after the final policy year there is no future benefit from the insurance policy and the insured is not incentivised to adopt the self-mitigation measure, a higher deductible is used as the incentive in the final policy year. 
In addition, we let $\delta_{\Ti\Tn}(t)=0.75(t-16)^+$, $\delta_{\To\Tu\Tt}(t)=3+\frac{5}{19}(t-1)$, and $\delta_{\Tr\Te}=3$. This setting has the effect of incentivising the insured to activate the insurance contract early on, and disincentivising withdrawal when close to the final policy year. 
As a baseline for comparison, we also consider another cyber risk insurance policy without the Bonus-Malus system, which can be modelled by letting $\CB=\{0\}$. We fix the premium to be the base premium $p^{\CB\CM}_{\text{base}}$, and leave everything else identical to the policy with the Bonus-Malus system. 
Furthermore, in Algorithm~\ref{alg:fft}, we fix $\overline{l}=10000$, $K_{\Tg\Tr}=20$, $\theta=\frac{20}{2^{K_{\Tg\Tr}}}=3.0518\times10^{-4}$.

\subsection{Results and Discussion}

Figure~\ref{fig:exp1policy} shows the expected number of years the insured's policy spends in each of the Bonus-Malus levels or being uninsured and the expected number of years the insured adopts the self-mitigation measure. The two panels compare the cyber risk insurance contract with the Bonus-Malus system with the one without.
With the contract that does not have the Bonus-Malus system, the decisions of the insured are completely deterministic, that is, they do not depend on the realisation of the cyber loss events. When $p^{\CB\CM}_{\text{base}}\le4.410$, the optimal strategy of the insured is to purchase the cyber risk insurance policy every year and only adopt the self-mitigation measure in the final policy year (due to the higher deductible in the final policy year). When $p^{\CB\CM}_{\text{base}}\ge4.415$, the optimal strategy of the insured is to never purchase the cyber risk insurance policy and always adopt the self-mitigation measure. Therefore, without the Bonus-Malus system, the issue of moral hazard is present and the insured will treat the cyber risk insurance policy and the self-mitigation measure as substitute goods. 
On the other hand, when the Bonus-Malus system is introduced to the cyber risk insurance contract, the decisions of the insured depend on the realisation of the cyber loss events. When $4.495\le p^{\CB\CM}_{\text{base}}\le4.930$, the optimal strategy of the insured is to always purchase the cyber risk insurance policy and adopt the self-mitigation measure. When $4.935\le p^{\CB\CM}_{\text{base}}\le5.050$, the optimal strategy of the insured is to always adopt the self-mitigation measure but withdraw from the insurance contract when the expected future cost exceeds the expected future benefit of the insurance policy. As a result, the retention rate, i.e., the expected proportion of years the insured remains in the contract, drops when the base premium is increased. When $p^{\CB\CM}_{\text{base}}\ge5.055$, the optimal strategy of the insured is to never purchase the cyber risk insurance policy and always adopt the self-mitigation measure. Hence, compared with the contract without Bonus-Malus, the contract with Bonus-Malus incentivises the insured to adopt the self-mitigation measure in addition to purchasing the cyber risk insurance policy. 

\begin{figure}[t]
\centering
\includegraphics[width=0.48\linewidth]{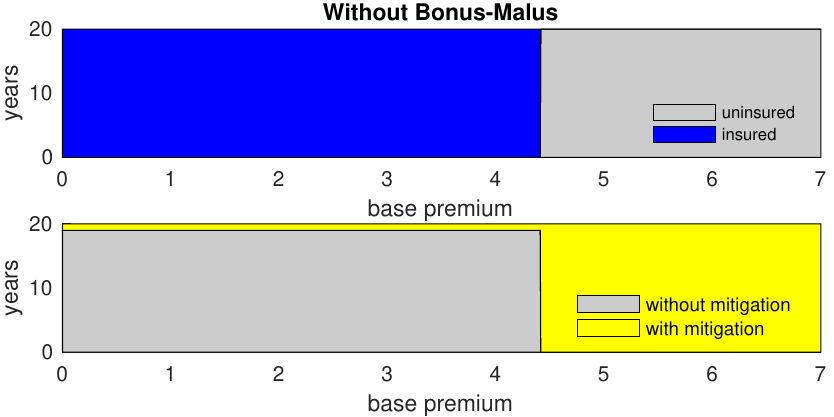}
~
\includegraphics[width=0.48\linewidth]{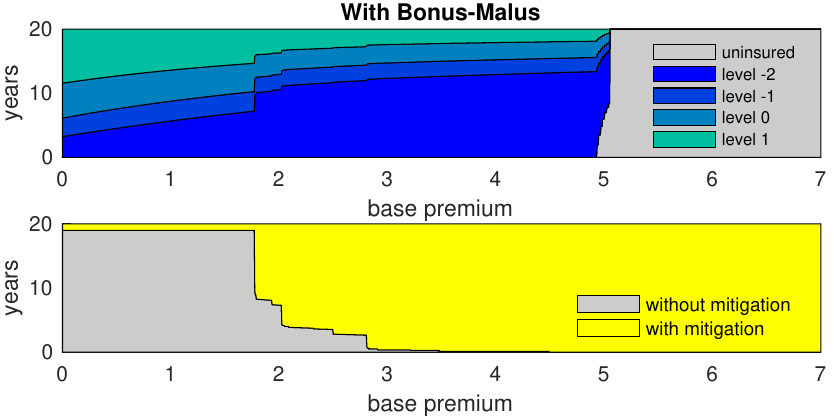}
\caption{The retention rate of the cyber risk insurance policy and the expected years of adoption of the self-mitigation measure versus the base premium.}
\label{fig:exp1policy}
\end{figure}

\begin{figure}[t]
\centering
\includegraphics[width=0.48\linewidth]{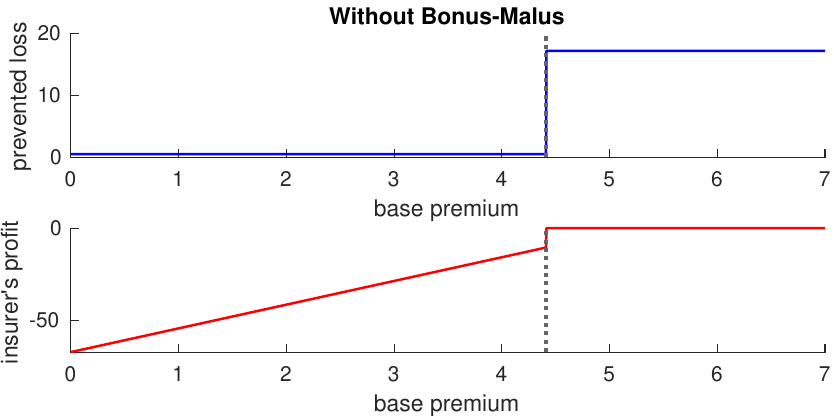}
~
\includegraphics[width=0.48\linewidth]{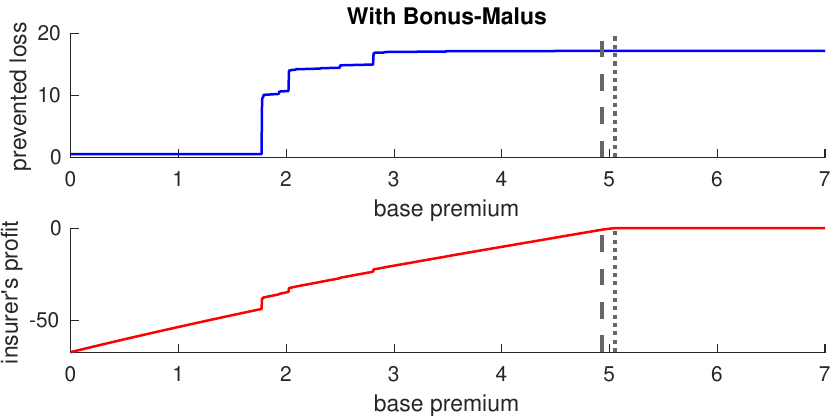}
\caption{The discounted total expected loss prevented by the self-mitigation measure and the discounted expected profit (measured by the quantity $\overline{Z}_{\Ti\Tn\Ts}-\overline{Z}_{\Tc\Tp}$) of the insurer versus the base premium. Left panel: the contract without the Bonus-Malus system. The dashed lines indicate the highest base premium before the insured chooses not to purchase the cyber risk insurance policy. Right panel: the contract with the Bonus-Malus system. The dashed lines indicate the highest base premium before the retention rate drops below 100\%. The dotted lines indicate the highest base premium before the insured chooses not to purchase the cyber risk insurance policy.}
\label{fig:exp1profit}
\end{figure}

Figure~\ref{fig:exp1profit} compares both the expected value of the discounted total loss prevented by the self-mitigation measure and the expected value of the discounted profit of the insurer measured by the quantity $\overline{Z}_{\Ti\Tn\Ts}-\overline{Z}_{\Tc\Tp}$ (defined in Section~\ref{ssec:pricing}) in the two policies. 
As we have discussed in Section~\ref{ssec:pricing}, the quantity $\overline{Z}_{\Ti\Tn\Ts}-\overline{Z}_{\Tc\Tp}$ is always non-positive and its value will be 0 when the insurance policy is not purchased.
Moreover, recall from the discussion in Section~\ref{ssec:pricing} that equating $\overline{Z}_{\Ti\Tn\Ts}-\overline{Z}_{\Tc\Tp}$ with the insurer's expected profit incurs a number of practical concerns.
The left panel of Figure~\ref{fig:exp1profit} shows the case without Bonus-Malus. In that case, when $p^{\CB\CM}_{\text{base}}\le4.410$, the insured will always purchase the cyber risk insurance policy but will only adopt the self-mitigation measure in the final policy year. Hence, the discounted total expected loss prevented stays at $0.505$, while the discounted expected profit of the insurer increases as the base premium increases. When $p^{\CB\CM}_{\text{base}}\ge4.415$, the insured will not purchase the insurance policy but will always adopt the self-mitigation measure. 
As a result, the discounted total expected loss prevented will be $17.183$ but the value of $\overline{Z}_{\Ti\Tn\Ts}-\overline{Z}_{\Tc\Tp}$ will be equal to~0 since the insurance policy is never purchased.
The most the insurer can gain before losing the insured is $-10.510$, when the base premium is set to $4.410$. 
In contrast, in the case with the Bonus-Malus system, as shown in the right panel of Figure~\ref{fig:exp1profit}, the insurer can gain a discounted expected profit of at most $-0.860$ while always retaining the insured (i.e., the insured will never withdraw from the contract), when the base premium is set to $4.930$. The insurer can gain a discounted expected profit of at most $-0.006$ before losing the insured, when the base premium is set to $5.050$. With both of these base premiums, the insured will always adopt the self-mitigation measure, resulting in a discounted total expected loss prevention of~$17.183$. 

Overall, this experiment demonstrates two benefits of the Bonus-Malus system. First, the presence of the Bonus-Malus system in the cyber risk insurance contract incentivises the insured to adopt the self-mitigation measure. This results in a considerable increase in the prevention of cyber losses, which enhances the overall security of the cyberspace. Second, the Bonus-Malus system benefits the insurer, since it allows the insurer to gain more profit from the cyber risk insurance policy while remaining attractive to the insured. 
To show that the observations from the results in this experiment and the conclusions drawn do not depend on the specific choice of the $h$ parameter in the loss severity distribution (which is the most impactful parameter in the truncated g-and-h distribution), and that they also do not depend on our distributional assumption, we have repeated the same experiment with slightly modified settings.
In the first modified setting, the $h$ parameter in the truncated g-and-h distribution is set to $0.10$, $0.20$, or $0.25$. 
In the second modified setting, the loss severity distribution is replaced by a log-normal distribution where the parameters are determined by matching the first two moments to $\text{Tr-}g\text{-and-}h(\mu=0,\varsigma=1,g=1.8,h=0.15)$. 
The results obtained under these modified settings turned out to be very similar to the results in the original experiment.\footnote{The results under these modified settings are available in the online appendix on GitHub: \url{https://github.com/qikunxiang/CyberInsuranceBonusMalus}}
This shows that the benefits of the Bonus-Malus system in the experiment are in fact present across various loss distributional assumptions.

\section{Conclusion}
\label{sec:conclusion}

This paper motivated the joint consideration of risk reduction and risk transfer decisions in the face of cyber threats. We introduced a cyber risk insurance contract with a Bonus-Malus system to provide incentive mechanisms to promote the adoption of cyber risk mitigation practices. We developed a model based on the stochastic optimal control framework to analyse how a rational insured allocates funds between self-mitigation measures and a cyber risk insurance policy. A dynamic programming-based algorithm was then developed to efficiently solve this decision problem. 
A numerical experiment demonstrated that this novel type of insurance contract can incentivise the adoption of self-mitigation measures and can allow the insurer to profit more from the policy while remaining attractive to the insured. 
Future research could investigate the effects of the risk profile, i.e., the characteristics of the loss distribution such as the heaviness of its tail, on the effectiveness of the Bonus-Malus system and how one can tailor Bonus-Malus-based insurance contracts for different risk profiles.  
Even though currently such Bonus-Malus systems for cyber risk insurance are conceptual and not yet adopted by the cyber insurers in the market, the purpose of this paper is to demonstrate that such cyber risk insurance contracts can be considered and further tailored for practice.
Moreover, we hope that the increasing availability of high-quality cyber datasets and further research studies can help to further bridge this gap in the future.
In particular, more research studies like \citep{skarczinski2022more} are needed in order to quantitatively investigate the link between cybersecurity countermeasures and cyber losses, and to aid the joint consideration of cyber risk reduction and risk transfer decisions. 
As suggested by \citet{chase2017scalable}, it might be mutually beneficial for cyber insurers to work more closely with Security-as-a-Service providers.

\section*{Acknowledgments}
\noindent
Ariel Neufeld gratefully acknowledges the financial support by his Nanyang Assistant Professorship Grant (NAP Grant) \emph{Machine Learning based Algorithms in Finance and Insurance}.

\appendix
\section{Proofs}
\label{apx:proofs}

\begin{proof}[Proof of Theorem~\ref{thm:dp}]
In this proof, we apply the dynamic programming principle and perform backward induction in time to show the optimality of $\pi^\star$.
First, one may check that $d^\star_t,\iota^\star_t$ are $\CF_{t-1}$-measurable, $j^\star_t$ is $\CF_t$-measurable, and ${\{\iota^\star_t=0,j^\star_t=1\}=\emptyset}$ for $t=1,\ldots,T$. Thus, indeed $\pi^\star\in\Pi$.
For all $\pi=(d_s,\iota_s,j_s)_{s=1:T}\in\Pi$ and $t\in\{0,\ldots,T\}$, let us define $O_t(\pi)=(\widetilde{d}_s,\widetilde{\iota}_s,\widetilde{j}_s)_{s=1:T}\in\Pi$ as follows:
\begin{align}
\begin{split}
&\hspace{-1.55cm}\big(b^{O_t(\pi)}_{0},i^{O_t(\pi)}_{0}\big):=(0,\Tn\To),\\
\text{for each }s&=1,\ldots,t, \text{ let:}\\
&\widetilde{d}_s=d_s,\hspace{2.92cm}\widetilde{\iota}_s=\iota_s, \hspace{2.92cm}\widetilde{j}_s=j_s,\\
\text{for each }s&=t+1,\ldots,T, \text{ let:}\\
&\widetilde{d}_s=\widehat{d}_s\big(b_{s-1}^{O_t(\pi)},i_{s-1}^{O_t(\pi)}\big),\quad\quad \widetilde{\iota}_s=\widehat{\iota}_s\big(b_{s-1}^{O_t(\pi)},i_{s-1}^{O_t(\pi)}\big),\quad\quad \widetilde{j}_s=\widehat{j}_s\big(b_{s-1}^{O_t(\pi)},i_{s-1}^{O_t(\pi)},W_s\big).
\end{split}
\label{eqn:dpconcept-process}
\end{align}
By the definition above, one may check that $O_t(\pi)\in\Pi$ for all $\pi\in\Pi$ and $t=0,\ldots,T$.
In particular, when $t=0$, (\ref{eqn:dpconcept-process}) implies that $O_0(\pi)=\pi^\star$ for all $\pi\in\Pi$. In addition, notice that $O_{t+s}\big(O_t(\pi)\big)=O_t(\pi)$ for all $\pi\in\Pi$ and $s\ge0$. 

Next, we prove the following statement by induction: 
\begin{align}
\begin{split}
\text{for all }t=0,\ldots,T,\quad\CV_t\big(b_{t}^{O_t(\pi)},i_{t}^{O_t(\pi)}\big)=V^{O_t(\pi)}_t\le V^\pi_t\quad \PROB\text{-a.s.\ for all }\pi\in\Pi.
\end{split}
\label{eqn:dpinduction}
\end{align}
To begin, we have by definition that $O_T(\pi)=\pi$ and $\CV_T(b^{\pi}_{T},i^{\pi}_{T})=V^{\pi}_T=0$ $\PROB$-a.s.\ for all $\pi\in\Pi$. Hence, (\ref{eqn:dpinduction}) holds when $t=T$. 
Now, let us suppose that for some $t\in\{1,\ldots,T\}$, it holds that $\CV_t\big(b_{t}^{O_t(\pi)},i_{t}^{O_t(\pi)}\big)=V^{O_t(\pi)}_t\le V^\pi_t$ $\PROB$-a.s.\ for all $\pi\in\Pi$. 
Let $\pi=(d_s,\iota_s,j_s)_{s=1:T}\in\Pi$ be arbitrary and let $O_{t-1}(\pi):=(\widetilde{d}_s,\widetilde{\iota}_s,\widetilde{j}_s)_{s=1:T}\in\Pi$. 
Note that (\ref{eqn:dpconcept-j}) ensures that for every $b\in\CB,i\in\CI,w\in\CW$,
\begin{align}
\begin{split}
&\phantom{=}\;\;g_t(b,i,\widehat{d}_t(b,i),\widehat{\iota}_t(b,i),\widehat{j}_t(b,i,w),w)+\CV_t\big(f_t(b,i,\widehat{d}_t(b,i),\widehat{\iota}_t(b,i),\widehat{j}_t(b,i,w),w)\big)\\
&=\min_{j\in\{0,1\}}\Big\{g_t(b,i,\widehat{d}_t(b,i),\widehat{\iota}_t(b,i),j,w)+\CV_t\big(f_t(b,i,\widehat{d}_t(b,i),\widehat{\iota}_t(b,i),j,w)\big)\Big\}.
\end{split}
\label{eqn:dpproof-j-optimality}
\end{align}
Combining (\ref{eqn:dpproof-j-optimality}) with the definition of $\widehat{d}_t$ and $\widehat{\iota}_t$ in (\ref{eqn:dpconcept-d-iota}), one can show that for all $b\in\CB$, $i\in\CI$, $d\in\CD$, $\iota\in\{0,1\}$, and $\FB(\CW)$-measurable $j:\CW\to\{0,1\}$, it holds that
\begin{align}
\begin{split}
&\phantom{=}\;\;\CV_{t-1}(b,i)\\
&=e^{-r}\EXP\bigg[g_{t}(b,i,\widehat{d}_t(b,i),\widehat{\iota}_t(b,i),\widehat{j}_t(b,i,W_t),W_t)+\CV_{t}\big(f_{t}(b,i,\widehat{d}_t(b,i),\widehat{\iota}_t(b,i),\widehat{j}_t(b,i,W_t),W_t)\big)\bigg]\\
&\le e^{-r}\EXP\bigg[g_{t}(b,i,d,\iota,j(W_t),W_t)+\CV_{t}\big(f_{t}(b,i,d,\iota,j(W_t),W_t)\big)\bigg]. 
\end{split}
\label{eqn:dpoptimality}
\end{align}
By (\ref{eqn:dpoptimality}), (\ref{eqn:dpconcept-process}), the independence between $\CF_{t-1}$ and $\sigma(W_t)$, and the induction hypothesis, it holds $\PROB$-a.s.\ that
\begin{align}
\begin{split}
&\phantom{=}\;\;\CV_{t-1}\big(b^{O_{t-1}(\pi)}_{t-1},i^{O_{t-1}(\pi)}_{t-1}\big)\\
&=e^{-r}\EXP\bigg[g_{t}(b,i,\widehat{d}_t(b,i),\widehat{\iota}_t(b,i),\widehat{j}_t(b,i,W_t),W_t)\\
&\qquad\qquad\quad+\CV_{t}\big(f_{t}(b,i,\widehat{d}_t(b,i),\widehat{\iota}_t(b,i),\widehat{j}_t(b,i,W_t),W_t)\big)\bigg]\bigg|_{\substack{b=b^{O_{t-1}(\pi)}_{t-1},\;i=i^{O_{t-1}(\pi)}_{t-1}}}\\
&=e^{-r}\EXP\bigg[g_{t}(b^{O_{t-1}(\pi)}_{t-1},i^{O_{t-1}(\pi)}_{t-1},\widehat{d}_t(b^{O_{t-1}(\pi)}_{t-1},i^{O_{t-1}(\pi)}_{t-1}),\widehat{\iota}_t(b^{O_{t-1}(\pi)}_{t-1},i^{O_{t-1}(\pi)}_{t-1}),\\
&\qquad\qquad\quad\widehat{j}_t(b^{O_{t-1}(\pi)}_{t-1},i^{O_{t-1}(\pi)}_{t-1},W_t),W_t)+\CV_{t}\big(f_{t}(b^{O_{t-1}(\pi)}_{t-1},i^{O_{t-1}(\pi)}_{t-1},\widehat{d}_t(b^{O_{t-1}(\pi)}_{t-1},i^{O_{t-1}(\pi)}_{t-1}),\\
&\qquad\qquad\quad\widehat{\iota}_t(b^{O_{t-1}(\pi)}_{t-1},i^{O_{t-1}(\pi)}_{t-1}),\widehat{j}_t(b^{O_{t-1}(\pi)}_{t-1},i^{O_{t-1}(\pi)}_{t-1},W_t),W_t)\big)\bigg|\CF_{t-1}\bigg]\\
&=e^{-r}\EXP\bigg[g_{t}\big(b^{O_{t-1}(\pi)}_{t-1},i^{O_{t-1}(\pi)}_{t-1},\widetilde{d}_t,\widetilde{\iota}_t,\widetilde{j}_t,W_t\big)+\CV_{t}\big(f_{t}\big(b^{O_{t-1}(\pi)}_{t-1},i^{O_{t-1}(\pi)}_{t-1},\widetilde{d}_t,\widetilde{\iota}_t,\widetilde{j}_t,W_t\big)\big)\bigg|\CF_{t-1}\bigg]\\
&=e^{-r}\EXP\bigg[g_{t}\big(b^{O_{t-1}(\pi)}_{t-1},i^{O_{t-1}(\pi)}_{t-1},\widetilde{d}_t,\widetilde{\iota}_t,\widetilde{j}_t,W_t\big)+\CV_{t}\big(b^{O_t(O_{t-1}(\pi))}_{t},i^{O_t(O_{t-1}(\pi))}_{t}\big)\bigg|\CF_{t-1}\bigg]\\
&=e^{-r}\EXP\bigg[g_{t}\big(b^{O_{t-1}(\pi)}_{t-1},i^{O_{t-1}(\pi)}_{t-1},\widetilde{d}_t,\widetilde{\iota}_t,\widetilde{j}_t,W_t\big)+V^{O_{t-1}(\pi)}_{t}\bigg|\CF_{t-1}\bigg]\\
&=V^{O_{t-1}(\pi)}_{t-1}.
\end{split}
\label{eqn:dpproof-step1}
\end{align}

Now, let $\pi=(d_s,\iota_s,j_s)_{s=1:T}\in\Pi$ be arbitrary and let $O_{t-1}(\pi):=(\widetilde{d}_s,\widetilde{\iota}_s,\widetilde{j}_s)_{s=1:T}\in\Pi$. 
By (\ref{eqn:dpconcept-process}), we have 
\begin{align}
\big(b^\pi_{t-1},i^\pi_{t-1}\big)=\big(b^{O_{t-1}(\pi)}_{t-1},i^{O_{t-1}(\pi)}_{t-1}\big)=\big(b^{O_t(\pi)}_{t-1},i^{O_t(\pi)}_{t-1}\big)\quad\PROB\text{-a.s.}
\label{eqn:dpproof-identity}
\end{align}
By (\ref{eqn:dpproof-step1}), (\ref{eqn:dpproof-j-optimality}), (\ref{eqn:dpproof-identity}), the independence between $\CF_{t-1}$ and $\sigma(W_t)$, and the induction hypothesis, it holds $\PROB$-a.s.\ that
\begin{align}
\begin{split}
&\phantom{=}\;\;V^{O_{t-1}(\pi)}_{t-1}\\
&=e^{-r}\EXP\bigg[\min_{j\in\{0,1\}}\Big\{g_{t}(b,i,\widehat{d}_t(b,i),\widehat{\iota}_t(b,i),j,W_t)+\CV_{t}\big(f_{t}(b,i,\widehat{d}_t(b,i),\widehat{\iota}_t(b,i),j,W_t)\big)\Big\}\bigg]\bigg|_{\substack{b=b^{O_{t-1}(\pi)}_{t-1},\;i=i^{O_{t-1}(\pi)}_{t-1}}}\\
&\le e^{-r}\EXP\bigg[\min_{j\in\{0,1\}}\Big\{g_{t}(b,i,d_t,\iota_t,j,W_t)+\CV_{t}\big(f_{t}(b,i,d_t,\iota_t,j,W_t)\big)\Big\}\bigg]\bigg|_{\substack{b=b^{O_{t-1}(\pi)}_{t-1},\;i=i^{O_{t-1}(\pi)}_{t-1}}}\\
&=e^{-r}\EXP\bigg[\min_{j\in\{0,1\}}\Big\{g_{t}\big(b^{O_{t-1}(\pi)}_{t-1},i^{O_{t-1}(\pi)}_{t-1},d_t,\iota_t,j,W_t\big)+\CV_{t}\big(f_{t}\big(b^{O_{t-1}(\pi)}_{t-1},i^{O_{t-1}(\pi)}_{t-1},d_t,\iota_t,j,W_t\big)\big)\Big\}\bigg|\CF_{t-1}\bigg]\\
&\le e^{-r}\EXP\bigg[g_{t}\big(b^{O_{t-1}(\pi)}_{t-1},i^{O_{t-1}(\pi)}_{t-1},{d}_t,{\iota}_t,{j}_t,W_t\big)+\CV_{t}\big(f_{t}\big(b^{O_{t-1}(\pi)}_{t-1},i^{O_{t-1}(\pi)}_{t-1},{d}_t,{\iota}_t,{j}_t,W_t\big)\big)\bigg|\CF_{t-1}\bigg]\\
&=e^{-r}\EXP\bigg[g_{t}\big(b^{\pi}_{t-1},i^{\pi}_{t-1},{d}_t,{\iota}_t,{j}_t,W_t\big)+\CV_{t}\big(b^{O_t(\pi)}_{t},i^{O_t(\pi)}_{t}\big)\bigg|\CF_{t-1}\bigg]\\
&\le e^{-r}\EXP\bigg[g_{t}\big(b^{\pi}_{t-1},i^{\pi}_{t-1},{d}_t,{\iota}_t,{j}_t,W_t\big)+V_t^{\pi}\bigg|\CF_{t-1}\bigg]\\
&=V_{t-1}^{\pi}.
\end{split}
\label{eqn:dpproof-step2}
\end{align}
Combining (\ref{eqn:dpproof-step1}) and (\ref{eqn:dpproof-step2}), we have shown that $\CV_{t-1}\big(b_{t-1}^{O_{t-1}(\pi)},i_{t-1}^{O_{t-1}(\pi)}\big)=V^{O_{t-1}(\pi)}_{t-1}\le V^\pi_{t-1}$ $\PROB$-a.s.\ for all $\pi\in\Pi$. 
By induction, (\ref{eqn:dpinduction}) holds for $t=0$. 
Hence, $\CV_0\big(b^{\pi^\star}_0,i^{\pi^\star}_0\big)=V^{\pi^\star}_0\le V^{\pi}_0$ for all $\pi\in\Pi$. Since $(b^{\pi^\star}_0,i^{\pi^\star}_0\big)=(0,\Tn\To)$, we have $\CV_0\big(b^{\pi^\star}_0,i^{\pi^\star}_0\big)=V^{\pi^\star}_0=\inf_{\pi\in\Pi}V^{\pi}_0=\BBV_0$.
The proof is now complete.
\end{proof}

\begin{proof}[Proof of Theorem~\ref{thm:dpalgo}]
To prove statement~\ref{sthm:dpalgo1}, it suffices to show that \underline{Claim~1} and \underline{Claim~2} in Remark~\ref{rmk:dpalgo} hold true. 
If these two claims hold, then one can verify that $\widehat{d}_t(b,i),\widehat{\iota}_t(b,i),\widehat{j}_t(b,i,w)$ defined on Line~\ref{alglin:dp-d-iota-def} and Line~\ref{alglin:dp-j-def} coincide with the definitions (\ref{eqn:dpconcept-d-iota}) and (\ref{eqn:dpconcept-j}), and thus statement~\ref{sthm:dpalgo1} holds as a consequence of Theorem~\ref{thm:dp}.
In \underline{Claim~1}, by (\ref{eqn:socstate}) and (\ref{eqn:soccost}), it holds that
\begin{align}
\begin{split}
g_{t}(b,i,d,1,1,w)+\CV_{t}\big(f_{t}(b,i,d,1,1,w)\big)&< g_{t}(b,i,d,1,0,w)+\CV_{t}\big(f_{t}(b,i,d,1,0,w)\big)\\
&\Updownarrow\\
\CV_{t}\big(\CB\CM(b,\lambda^{\CB\CM}(b,t,L(d,w)),\To\Tn\big)-\lambda^{\CB\CM}(b,t,L(d,w))&<\CV_{t}\big(\CB\CM(b,0),\To\Tn\big).
\end{split}
\label{eqn:dpalgo-claim1-proof1}
\end{align}
Observe that, by the definitions of $\underline{b},\overline{b}$ on Line~\ref{alglin:dp-b-b-def} and the definition of $\alpha_t(b,b')$ on Line~\ref{alglin:dp-L-def}, we have
\begin{align}
\begin{split}
&\phantom{=}\;\;\Big\{c\in\R_+:\CV_{t}\big(\CB\CM(b,c),\To\Tn\big)-c<\CV_{t}\big(\CB\CM(b,0),\To\Tn\big)\Big\}\\
&=\bigcup_{\underline{b}\le b'\le\overline{b}}\Big\{c\in\R_+:\CB\CM(b,c)=b',c>\CV_{t}(b',\To\Tn)-\CV_{t}(\underline{b},\To\Tn)\Big\}\\
&=\bigcup_{\underline{b}\le b'\le\overline{b}}\CL_t(b,b').
\end{split}
\label{eqn:dpalgo-claim1-proof2}
\end{align}
Hence, \underline{Claim~1} holds true.
In \underline{Claim~2}, in the case where $\iota=1$, we have, by (\ref{eqn:socstate}) and (\ref{eqn:soccost}), that
\begin{align}
\begin{split}
&\phantom{=}\;\;\EXP\bigg[\min_{j\in\{0,1\}}\Big\{g_{t}(b,i,d,\iota,j,W_t)+\CV_{t}\big(f_{t}(b,i,d,\iota,j,W_t)\big)\Big\}\bigg]\\
&=\EXP\Big[\beta(d)+\iota p^{\CB\CM}(b,t)+\delta_{\Ti\Tn}(t)\INDI_{\{i=\Tn\To,\iota=1\}}+\delta_{\To\Tu\Tt}(t)\INDI_{\{i=\To\Tn,\iota=0\}}+\delta_{\Tr\Te}\INDI_{\{i\ne\To\Tn,i\ne\Tn\To,\iota=1\}}+L(d,W_t)\Big]\\
&\qquad+\EXP\bigg[\min_{j\in\{0,1\}}\Big\{\CV_{t}\big(\CB\CM(b,j\lambda^{\CB\CM}(b,t,L(d,W_t))),\To\Tn\big)-\iota j\lambda^{\CB\CM}(b,t,L(d,W_t))\Big\}\bigg]\\
&=\beta(d)+\iota p^{\CB\CM}(b,t)+\delta_{\Ti\Tn}(t)\INDI_{\{i=\Tn\To,\iota=1\}}+\delta_{\To\Tu\Tt}(t)\INDI_{\{i=\To\Tn,\iota=0\}}+\delta_{\Tr\Te}\INDI_{\{i\ne\To\Tn,i\ne\Tn\To,\iota=1\}}+\EXP\big[L(d,W_t)\big]\\
&\qquad+\EXP\Big[\CV_{t}\big(\CB\CM(b,0),\To\Tn\big) \wedge \Big(\CV_{t}\big(\CB\CM(b,\lambda^{\CB\CM}(b,t,L(d,W_t))),\To\Tn\big)-\lambda^{\CB\CM}(b,t,L(d,W_t))\Big)\Big].
\end{split}
\label{eqn:dpalgo-claim2-proof1}
\end{align}
Moreover, by the definition of $H_t(b,i,d,\iota)$ on Line~\ref{alglin:dp-H-def1}, the definitions of $\underline{b},\overline{b}$ on Line~\ref{alglin:dp-b-b-def}, and the definition of $\alpha_t(b,b')$ on Line~\ref{alglin:dp-L-def}, we have
\begin{align}
\begin{split}
&\phantom{=}\;\;\EXP\Big[\CV_{t}\big(\CB\CM(b,0),\To\Tn\big) \wedge \Big(\CV_{t}\big(\CB\CM(b,\lambda^{\CB\CM}(b,t,L(d,W_t))),\To\Tn\big)-\lambda^{\CB\CM}(b,t,L(d,W_t))\Big)\Big]\\
&=\CV_{t}\big(\CB\CM(b,0),\To\Tn\big)\\
&\qquad-\EXP\bigg[\Big(\lambda^{\CB\CM}(b,t,L(d,W_t))-\CV_{t}\big(\CB\CM(b,\lambda^{\CB\CM}(b,t,L(d,W_t))),\To\Tn\big)+\CV_{t}\big(\CB\CM(b,0),\To\Tn\big)\Big)^+\bigg]\\
&=\CV_{t}(\underline{b},\To\Tn)\\
&\qquad-\sum_{\underline{b}\le b'\le\overline{b}}\EXP\bigg[\INDI_{\{\CB\CM(b,\lambda^{\CB\CM}(b,t,L(d,W_t)))=b'\}}\Big(\lambda^{\CB\CM}(b,t,L(d,W_t))-\big[\CV_{t}(b',\To\Tn)-\CV_{t}(\underline{b},\To\Tn)\big]\Big)^+\bigg]\\
&=H_t(b,i,d,1).
\end{split}
\label{eqn:dpalgo-claim2-proof2}
\end{align}
In the case where $\iota=0$, \underline{Claim~2} follows directly from (\ref{eqn:socstate}) and (\ref{eqn:soccost}). 
Thus, \underline{Claim~2} holds true.

Now, let us prove statement~\ref{sthm:dpalgo2}. Let $(b,i)\in\CB\times\CI$ be fixed, let $\underline{b},\overline{b}$ be defined by Line~\ref{alglin:dp-b-b-def}, and let $\CL_t(b,b')$ be defined by Line~\ref{alglin:dp-L-def}. It follows from (\ref{eqn:socstate}) that, if $\widehat{\iota}_t(b,i)=0$, then
\begin{align*}
\PROB\big[\big(b^{\pi^\star}_{t},i^{\pi^\star}_{t}\big)=\CB\CM_0(b,i)\big|b^{\pi^\star}_{t-1}=b,i^{\pi^\star}_{t-1}=i\big]=1=P^\star_t\big[(b,i)\rightarrow\CB\CM_0(b,i)\big],
\end{align*}
thus showing the correctness of Line~\ref{alglin:dp-kern3}. 
Now, suppose that $\widehat{\iota}_t(b,i)=1$. Then, by (\ref{eqn:socstate}), we have that $\PROB\big[i^{\pi^\star}_{t}=\To\Tn\big|b^{\pi^\star}_{t-1}=b,i^{\pi^\star}_{t-1}=i\big]=1$. Let us first examine the case where $b'\ne\underline{b}$. We then have
\begin{align*}
&\phantom{=}\;\;\Big\{\big(b^{\pi^\star}_{t},i^{\pi^\star}_{t}\big)=(b',\To\Tn), \big(b^{\pi^\star}_{t-1},i^{\pi^\star}_{t-1}\big)=(b,i) \Big\}\\
&=\Big\{\CB\CM\Big(b,\lambda^{\CB\CM}(b,t,L(d^\star_t,W_t))\Big)=b'\Big\}\cap\big\{j^\star_t=1\big\} \cap \Big\{\big(b^{\pi^\star}_{t-1},i^{\pi^\star}_{t-1}\big)=(b,i)\Big\}\\
&=\Big\{\CB\CM\Big(b,\lambda^{\CB\CM}(b,t,L(\widehat{d}_t(b,i),W_t))\Big)=b'\Big\}\cap\Big\{\widehat{j}_t(b,i,W_t)=1\Big\}\cap \Big\{\big(b^{\pi^\star}_{t-1},i^{\pi^\star}_{t-1}\big)=(b,i)\Big\}\\
&=\Big\{\CB\CM\Big(b,\lambda^{\CB\CM}(b,t,L(\widehat{d}_t(b,i),W_t))\Big)=b'\Big\} \cap \Big\{\lambda^{\CB\CM}(b,t,L(\widehat{d}_t(b,i),W_t))\in\textstyle\bigcup_{\underline{b}\le b''\le\overline{b}}\CL_t(b,b'')\Big\} \\
&\qquad\cap \Big\{\big(b^{\pi^\star}_{t-1},i^{\pi^\star}_{t-1}\big)=(b,i)\Big\}\\
&=\Big\{\lambda^{\CB\CM}(b,t,L(\widehat{d}_t(b,i),W_t))\in\CL_t(b,b')\Big\} \cap \Big\{\big(b^{\pi^\star}_{t-1},i^{\pi^\star}_{t-1}\big)=(b,i)\Big\},
\end{align*}
where the first equality is by (\ref{eqn:socstate}), the second equality is by (\ref{eqn:dpconcept-opt-process}), the third equality is by \underline{Claim~1} in Remark~\ref{rmk:dpalgo}, and the last equality is by Line~\ref{alglin:dp-L-def} and the property that $\{\CL_t(b,b''):\underline{b}\le b''\le\overline{b}\}$ are disjoint sets. Therefore, since $\big(b^{\pi^\star}_{t-1},i^{\pi^\star}_{t-1}\big)$ and $W_t$ are independent, we have for any $b'\ne\underline{b}$ that
\begin{align*}
\PROB\Big[b^{\pi^\star}_{t}=b',i^{\pi^\star}_{t}=\To\Tn,b^{\pi^\star}_{t-1}=b,i^{\pi^\star}_{t-1}=i\Big]=\PROB\Big[\lambda^{\CB\CM}(b,t,L(\widehat{d}_t(b,i),W_t))\in\CL_t(b,b')\Big]\PROB\Big[b^{\pi^\star}_{t-1}=b,i^{\pi^\star}_{t-1}=i\Big].
\end{align*}
Hence, by Line~\ref{alglin:dp-kern1},
\begin{align*}
\PROB\Big[b^{\pi^\star}_{t}=b',i^{\pi^\star}_{t}=\To\Tn\Big|b^{\pi^\star}_{t-1}=b,i^{\pi^\star}_{t-1}=i\Big]&=\PROB\Big[\lambda^{\CB\CM}(b,t,L(\widehat{d}_t(b,i),W_t))\in\CL_t(b,b')\Big]\\
&=P^\star_t\big[(b,i)\rightarrow(b',\To\Tn)\big].
\end{align*}
The remaining case where $b'=\underline{b}$ follows from
\begin{align*}
\PROB\Big[\big(b^{\pi^\star}_{t},i^{\pi^\star}_{t}\big)=(\underline{b},\To\Tn)\Big|\big(b^{\pi^\star}_{t-1},i^{\pi^\star}_{t-1}\big)=(b,i)\Big]&=1-\sum_{\underline{b}<b'\le\overline{b}}\PROB\Big[\big(b^{\pi^\star}_{t},i^{\pi^\star}_{t}\big)=(b',\To\Tn)\Big|\big(b^{\pi^\star}_{t-1},i^{\pi^\star}_{t-1}\big)=(b,i)\Big]\\
&=1-\sum_{\underline{b}<b'\le\overline{b}}P^\star_t\big[(b,i)\rightarrow(b',\To\Tn)\big]\\
&=P^\star_t\big[(b,i)\rightarrow(\underline{b},\To\Tn)\big],
\end{align*}
thus verifying the correctness of Line~\ref{alglin:dp-kern2}.
The correctness of Line~\ref{alglin:dp-marg-init} and Line~\ref{alglin:dp-marg-iter} follows from the definition that $\big(b^{\pi^\star}_{0},i^{\pi^\star}_{0}\big)=(0,\Tn\To)$ and basic properties of a finite state Markov chain. The proof of statement~\ref{sthm:dpalgo2} is complete. 

Finally, statement~\ref{sthm:dpalgo3} also follows from the basic properties of a finite state Markov chain. The proof is complete. 
\end{proof}

\begin{proof}[Proof of Lemma~\ref{lem:tr-g-and-h}]
Statement~(i) follows by checking the following:
\begin{align*}
\begin{split}
F_{X}(x)=\PROB[\widetilde{X}\le x|\widetilde{X}>0]=\frac{\PROB[0<\widetilde{X}\le x]}{\PROB[\widetilde{X}>0]}=\begin{cases}
\frac{F_{\widetilde{X}}(x)-F_{\widetilde{X}}(0)}{1-F_{\widetilde{X}}(0)} & \text{if }x>0,\\
0 & \text{if }x\le0.
\end{cases}
\end{split}
\end{align*}
Statement~(ii) can be verified directly by checking that $\PROB[X_U\le x]=F_X(x)$ for all $x\in\R$. 

Finally, statement~(iii) can be derived from (\ref{eqn:tr-g-and-h-df}) as follows:
\begin{align*}
\begin{split}
&\phantom{=}\;\;\EXP\big[(X-\gamma)^+\big]\\
&=\int_{\gamma}^{\infty}(x-\gamma) \DIFFM{F_{X}}{\DIFF x}\\
&=\frac{1}{1-F_{\widetilde{X}}(0)}\left[\int_{\gamma}^{\infty}x\DIFFM{F_{\widetilde{X}}}{\DIFF x}-\gamma(1-F_{\widetilde{X}}(\gamma))\right]\\
&=\frac{\varsigma}{1-F_{\widetilde{X}}(0)}\int_{Y_{g,h}^{-1}\left(\frac{\gamma-\mu}{\varsigma}\right)}^{\infty}Y_{g,h}(z)\DIFFM{\Phi}{\DIFF z}+\frac{(\mu-\gamma)(1-F_{\widetilde{X}}(\gamma))}{1-F_{\widetilde{X}}(0)}\\
&=\frac{\varsigma}{1-F_{\widetilde{X}}(0)}\frac{1}{g}\int_{Y_{g,h}^{-1}\left(\frac{\gamma-\mu}{\varsigma}\right)}^{\infty}(\exp(gz)-1)\exp\left(\frac{hz^2}{2}\right)\frac{1}{\sqrt{2\pi}}\exp\left(-\frac{z^2}{2}\right)\DIFFX{z}+\frac{(\mu-\gamma)(1-F_{\widetilde{X}}(\gamma))}{1-F_{\widetilde{X}}(0)}\\
&=\frac{\varsigma}{1-F_{\widetilde{X}}(0)}\frac{1}{g\sqrt{2\pi}}\Bigg[\int_{Y_{g,h}^{-1}\left(\frac{\gamma-\mu}{\varsigma}\right)}^{\infty}\exp\left(-\frac{(1-h)z^2}{2}+gz\right)-\exp\left(-\frac{(1-h)z^2}{2}\right)\DIFFX{z}\Bigg]+\frac{(\mu-\gamma)(1-F_{\widetilde{X}}(\gamma))}{1-F_{\widetilde{X}}(0)}\\
&=\frac{\varsigma}{1-F_{\widetilde{X}}(0)}\frac{1}{g\sqrt{2\pi}}\Bigg[\int_{Y_{g,h}^{-1}\left(\frac{\gamma-\mu}{\varsigma}\right)}^{\infty}\exp\left(\frac{g^2}{2(1-h)}\right)\exp\left(-\frac{(1-h)}{2}\left(z-\frac{g}{1-h}\right)^2\right)\DIFFX{z}\\
&\qquad\qquad\qquad\qquad\qquad\qquad\qquad-\int_{Y_{g,h}^{-1}\left(\frac{\gamma-\mu}{\varsigma}\right)}^{\infty}\exp\left(-\frac{(1-h)z^2}{2}\right)\DIFF{z}\Bigg]+\frac{(\mu-\gamma)(1-F_{\widetilde{X}}(\gamma))}{1-F_{\widetilde{X}}(0)}\\
&=\frac{\varsigma}{(1-F_{\widetilde{X}}(0))g\sqrt{1-h}}\Bigg[\exp\left(\frac{g^2}{2(1-h)}\right)\Phi\left(\left(\frac{g}{1-h}-Y_{g,h}^{-1}\left(\tfrac{\gamma-\mu}{\varsigma}\right)\right)\sqrt{1-h}\right)\\
&\qquad\qquad\qquad\qquad\qquad\qquad\qquad-\Phi\left(-Y_{g,h}^{-1}\left(\tfrac{\gamma-\mu}{\varsigma}\right)\sqrt{1-h}\right)\Bigg]+\frac{(\mu-\gamma)(1-F_{\widetilde{X}}(\gamma))}{1-F_{\widetilde{X}}(0)},
\end{split}
\end{align*}
where the last equality is obtained by noticing that both integrals are Gaussian integrals after a change of variable. 
The proof is now complete. 
\end{proof}

\bibliographystyle{abbrvnat}
\bibliography{references}

\begin{thebibliography}{81}
\providecommand{\natexlab}[1]{#1}
\providecommand{\url}[1]{\texttt{#1}}
\expandafter\ifx\csname urlstyle\endcsname\relax
  \providecommand{\doi}[1]{doi: #1}\else
  \providecommand{\doi}{doi: \begingroup \urlstyle{rm}\Url}\fi

\bibitem[Armenia et~al.(2021)Armenia, Angelini, Nonino, Palombi, and
  Schlitzer]{armenia2021dynamic}
S.~Armenia, M.~Angelini, F.~Nonino, G.~Palombi, and M.~F. Schlitzer.
\newblock A dynamic simulation approach to support the evaluation of cyber
  risks and security investments in {SME}s.
\newblock \emph{Decision Support Systems}, 147:\penalty0 113580, 2021.

\bibitem[Baione et~al.(2002)Baione, Levantesi, and
  Menzietti]{baione2002development}
F.~Baione, S.~Levantesi, and M.~Menzietti.
\newblock The development of an optimal bonus-malus system in a competitive
  market.
\newblock \emph{ASTIN Bulletin: The Journal of the IAA}, 32\penalty0
  (1):\penalty0 159--170, 2002.

\bibitem[Bandyopadhyay et~al.(2009)Bandyopadhyay, Mookerjee, and
  Rao]{bandyopadhyay2009managers}
T.~Bandyopadhyay, V.~S. Mookerjee, and R.~C. Rao.
\newblock Why {IT} managers don't go for cyber-insurance products.
\newblock \emph{Commun. ACM}, 52\penalty0 (11):\penalty0 68–73, Nov. 2009.

\bibitem[Bessy-Roland et~al.(2021)Bessy-Roland, Boumezoued, and
  Hillairet]{roland2021multivariate}
Y.~Bessy-Roland, A.~Boumezoued, and C.~Hillairet.
\newblock Multivariate {H}awkes process for cyber insurance.
\newblock \emph{Annals of Actuarial Science}, 15\penalty0 (1):\penalty0
  14–39, 2021.

\bibitem[Biener et~al.(2015)Biener, Eling, and Wirfs]{biener2015insurability}
C.~Biener, M.~Eling, and J.~H. Wirfs.
\newblock Insurability of cyber risk: An empirical analysis.
\newblock \emph{The Geneva Papers on Risk and Insurance-Issues and Practice},
  40\penalty0 (1):\penalty0 131--158, 2015.

\bibitem[B{\"{o}}hme and Schwartz(2010)]{bohme2010modeling}
R.~B{\"{o}}hme and G.~Schwartz.
\newblock Modeling cyber-insurance: Towards a unifying framework.
\newblock In \emph{9th Annual Workshop on the Economics of Information
  Security, {WEIS} 2010, Harvard University, Cambridge, MA, USA, June 7-8,
  2010}, 2010.

\bibitem[Boucher(2023)]{boucher2023bonus}
J.-P. Boucher.
\newblock Bonus-malus scale models: creating artificial past claims history.
\newblock \emph{Annals of Actuarial Science}, 17\penalty0 (1):\penalty0
  36–62, 2023.

\bibitem[Charpentier et~al.(2017)Charpentier, David, and
  Elie]{charpentier2017optimal}
A.~Charpentier, A.~David, and R.~Elie.
\newblock Optimal claiming strategies in bonus malus systems and implied
  {M}arkov chains.
\newblock \emph{Risks}, 5\penalty0 (4):\penalty0 58, 2017.

\bibitem[Chase et~al.(2019)Chase, Niyato, Wang, Chaisiri, and
  Ko]{chase2017scalable}
J.~Chase, D.~Niyato, P.~Wang, S.~Chaisiri, and R.~K.~L. Ko.
\newblock A scalable approach to joint cyber insurance and
  {S}ecurity-as-a-{S}ervice provisioning in cloud computing.
\newblock \emph{IEEE Transactions on Dependable and Secure Computing},
  16\penalty0 (4):\penalty0 565--579, 2019.

\bibitem[Craigen et~al.(2014)Craigen, Diakun-Thibault, and
  Purse]{craigen2014defining}
D.~Craigen, N.~Diakun-Thibault, and R.~Purse.
\newblock Defining cybersecurity.
\newblock \emph{Technology Innovation Management Review}, 4\penalty0
  (10):\penalty0 13--21, 2014.

\bibitem[Cremer et~al.(2022)Cremer, Sheehan, Fortmann, Kia, Mullins, Murphy,
  and Materne]{cremer2022cyber}
F.~Cremer, B.~Sheehan, M.~Fortmann, A.~N. Kia, M.~Mullins, F.~Murphy, and
  S.~Materne.
\newblock Cyber risk and cybersecurity: A systematic review of data
  availability.
\newblock \emph{The Geneva Papers on Risk and Insurance-Issues and Practice},
  47\penalty0 (3):\penalty0 698--736, 2022.

\bibitem[Cruz et~al.(2015)Cruz, Peters, and Shevchenko]{cruz2015fundamental}
M.~G. Cruz, G.~W. Peters, and P.~V. Shevchenko.
\newblock \emph{Fundamental aspects of operational risk and insurance
  analytics: A handbook of operational risk}.
\newblock John Wiley \& Sons, 2015.

\bibitem[Dacorogna et~al.(2023)Dacorogna, Debbabi, and
  Kratz]{dacorogna2023building}
M.~Dacorogna, N.~Debbabi, and M.~Kratz.
\newblock Building up cyber resilience by better grasping cyber risk via a new
  algorithm for modelling heavy-tailed data.
\newblock \emph{European Journal of Operational Research}, 2023.

\bibitem[Dou et~al.(2020)Dou, Tang, Wu, Qi, Xu, Zhang, and
  Hu]{dou2020insurance}
W.~Dou, W.~Tang, X.~Wu, L.~Qi, X.~Xu, X.~Zhang, and C.~Hu.
\newblock An insurance theory based optimal cyber-insurance contract against
  moral hazard.
\newblock \emph{Information Sciences}, 527:\penalty0 576--589, 2020.

\bibitem[Dutta and Perry(2006)]{dutta2006tale}
K.~Dutta and J.~Perry.
\newblock A tale of tails: an empirical analysis of loss distribution models
  for estimating operational risk capital.
\newblock Technical report, Federal Reserve Bank of Boston, 2006.

\bibitem[EIOPA(2018)]{eiopa2018}
EIOPA.
\newblock Understanding cyber insurance --- a structured dialogue with
  insurance companies.
\newblock
  \url{https://www.eiopa.europa.eu/publications/understanding-cyber-insurance-structured-dialogue-insurance-companies_en},
  Aug 2018.
\newblock Accessed: 2023-05-23.

\bibitem[Eling(2020)]{eling2020cyber}
M.~Eling.
\newblock Cyber risk research in business and actuarial science.
\newblock \emph{Eur. Actuar. J.}, 10\penalty0 (2):\penalty0 303--333, 2020.

\bibitem[Eling and Loperfido(2017)]{eling2017data}
M.~Eling and N.~Loperfido.
\newblock Data breaches: Goodness of fit, pricing, and risk measurement.
\newblock \emph{Insurance: mathematics and economics}, 75:\penalty0 126--136,
  2017.

\bibitem[Eling and Wirfs(2019)]{eling2019actual}
M.~Eling and J.~Wirfs.
\newblock What are the actual costs of cyber risk events?
\newblock \emph{European Journal of Operational Research}, 272\penalty0
  (3):\penalty0 1109--1119, 2019.

\bibitem[Eling and Wirfs(2015)]{eling2015modelling}
M.~Eling and J.~H. Wirfs.
\newblock Modelling and management of cyber risk.
\newblock \emph{International Actuarial Association Life Section}, 2015.

\bibitem[Embrechts and Frei(2009)]{embrechts2009panjer}
P.~Embrechts and M.~Frei.
\newblock Panjer recursion versus {FFT} for compound distributions.
\newblock \emph{Math. Methods Oper. Res.}, 69\penalty0 (3):\penalty0 497--508,
  2009.

\bibitem[Embrechts et~al.(1997)Embrechts, Kl\"{u}ppelberg, and
  Mikosch]{embrechts1997modelling}
P.~Embrechts, C.~Kl\"{u}ppelberg, and T.~Mikosch.
\newblock \emph{Modelling extremal events}, volume~33 of \emph{Applications of
  Mathematics (New York)}.
\newblock Springer-Verlag, Berlin, 1997.

\bibitem[Fahrenwaldt et~al.(2018)Fahrenwaldt, Weber, and
  Weske]{fahrenwaldt2018pricing}
M.~A. Fahrenwaldt, S.~Weber, and K.~Weske.
\newblock Pricing of cyber insurance contracts in a network model.
\newblock \emph{Astin Bull.}, 48\penalty0 (3):\penalty0 1175--1218, 2018.

\bibitem[Farkas et~al.(2021)Farkas, Lopez, and Thomas]{farkas2021cyber}
S.~Farkas, O.~Lopez, and M.~Thomas.
\newblock Cyber claim analysis using {G}eneralized {P}areto regression trees
  with applications to insurance.
\newblock \emph{Insurance Math. Econom.}, 98:\penalty0 92--105, 2021.

\bibitem[Feng et~al.(2018)Feng, Xiong, Niyato, Wang, and
  Leshem]{feng2018evolving}
S.~Feng, Z.~Xiong, D.~Niyato, P.~Wang, and A.~Leshem.
\newblock Evolving risk management against advanced persistent threats in fog
  computing.
\newblock In \emph{2018 IEEE 7th International Conference on Cloud Networking
  (CloudNet)}, pages 1--6, 2018.

\bibitem[Frachot et~al.(2001)Frachot, Georges, and Roncalli]{frachot2001loss}
A.~Frachot, P.~Georges, and T.~Roncalli.
\newblock Loss {D}istribution {A}pproach for operational risk.
\newblock \emph{Available at SSRN 1032523}, 2001.

\bibitem[Franco(2020)]{wef2020}
E.~G. Franco.
\newblock {The Global Risks Report} 2020, {World Economic Forum}.
\newblock \url{https://www.weforum.org/reports/the-global-risks-report-2020/},
  Jan. 2020.
\newblock Accessed: 2021-02-04.

\bibitem[G\'{o}mez-D\'{e}niz(2016)]{gomezdeniz2016bivariate}
E.~G\'{o}mez-D\'{e}niz.
\newblock Bivariate credibility bonus-malus premiums distinguishing between two
  types of claims.
\newblock \emph{Insurance Math. Econom.}, 70:\penalty0 117--124, 2016.

\bibitem[G{\'o}mez-D{\'e}niz and
  Calder{\'\i}n-Ojeda(2018)]{gomez2018multivariate}
E.~G{\'o}mez-D{\'e}niz and E.~Calder{\'\i}n-Ojeda.
\newblock Multivariate credibility in bonus-malus systems distinguishing
  between different types of claims.
\newblock \emph{Risks}, 6\penalty0 (2):\penalty0 34, 2018.

\bibitem[Gupta and Badve(2017)]{gupta2017taxonomy}
B.~B. Gupta and O.~P. Badve.
\newblock Taxonomy of {DoS} and {DDoS} attacks and desirable defense mechanism
  in a cloud computing environment.
\newblock \emph{Neural Comput. Appl.}, 28\penalty0 (12):\penalty0 3655–3682,
  Dec. 2017.

\bibitem[Herath et~al.(2020)Herath, Herath, and
  D'Arcy]{herath2020organizational}
T.~C. Herath, H.~S.~B. Herath, and J.~D'Arcy.
\newblock Organizational adoption of information security solutions: An
  integrative lens based on innovation adoption and the technology-
  organization- environment framework.
\newblock \emph{SIGMIS Database}, 51\penalty0 (2):\penalty0 12–35, May 2020.

\bibitem[Hoang et~al.(2017)Hoang, Wang, Niyato, and Hossain]{hoang2017charging}
D.~T. Hoang, P.~Wang, D.~Niyato, and E.~Hossain.
\newblock Charging and discharging of plug-in electric vehicles ({PEV}s) in
  vehicle-to-grid ({V2G}) systems: A cyber insurance-based model.
\newblock \emph{IEEE Access}, 5:\penalty0 732--754, 2017.

\bibitem[Holtan(2001)]{holtan2001optimal}
J.~Holtan.
\newblock Optimal insurance coverage under bonus-malus contracts.
\newblock \emph{ASTIN Bulletin: The Journal of the IAA}, 31\penalty0
  (1):\penalty0 175--186, 2001.

\bibitem[Hus{\'a}k et~al.(2019)Hus{\'a}k, Kom{\'a}rkov{\'a}, Bou-Harb, and
  {\v{C}}eleda]{husak2018survey}
M.~Hus{\'a}k, J.~Kom{\'a}rkov{\'a}, E.~Bou-Harb, and P.~{\v{C}}eleda.
\newblock Survey of attack projection, prediction, and forecasting in cyber
  security.
\newblock \emph{IEEE Communications Surveys \& Tutorials}, 21\penalty0
  (1):\penalty0 640--660, 2019.

\bibitem[Khalili et~al.(2018)Khalili, Naghizadeh, and
  Liu]{khalili2018designing}
M.~M. Khalili, P.~Naghizadeh, and M.~Liu.
\newblock Designing cyber insurance policies: The role of pre-screening and
  security interdependence.
\newblock \emph{IEEE Transactions on Information Forensics and Security},
  13\penalty0 (9):\penalty0 2226--2239, 2018.

\bibitem[Kshetri(2018)]{kshetri2018economics}
N.~Kshetri.
\newblock The economics of cyber-insurance.
\newblock \emph{IT Professional}, 20\penalty0 (6):\penalty0 9--14, 2018.

\bibitem[Lemaire(1995)]{lemaire1995bonus}
J.~Lemaire.
\newblock \emph{Bonus-malus systems in automobile insurance}, volume~19.
\newblock Springer science \& business media, 1995.

\bibitem[Lu et~al.(2018{\natexlab{a}})Lu, Niyato, Jiang, Wang, and
  Poor]{lu2018cyber}
X.~Lu, D.~Niyato, H.~Jiang, P.~Wang, and H.~V. Poor.
\newblock Cyber insurance for heterogeneous wireless networks.
\newblock \emph{IEEE Communications Magazine}, 56\penalty0 (6):\penalty0
  21--27, 2018{\natexlab{a}}.

\bibitem[Lu et~al.(2018{\natexlab{b}})Lu, Niyato, Privault, Jiang, and
  Wang]{lu2018managing}
X.~Lu, D.~Niyato, N.~Privault, H.~Jiang, and P.~Wang.
\newblock Managing physical layer security in wireless cellular networks: A
  cyber insurance approach.
\newblock \emph{IEEE Journal on Selected Areas in Communications}, 36\penalty0
  (7):\penalty0 1648--1661, 2018{\natexlab{b}}.

\bibitem[Maillart and Sornette(2010)]{maillart2010heavy}
T.~Maillart and D.~Sornette.
\newblock Heavy-tailed distribution of cyber-risks.
\newblock \emph{The European Physical Journal B}, 75\penalty0 (3):\penalty0
  357--364, 2010.

\bibitem[Malavasi et~al.(2022)Malavasi, Peters, Shevchenko, Tr{\"u}ck, Jang,
  and Sofronov]{malavasi2022cyber}
M.~Malavasi, G.~W. Peters, P.~V. Shevchenko, S.~Tr{\"u}ck, J.~Jang, and
  G.~Sofronov.
\newblock Cyber risk frequency, severity and insurance viability.
\newblock \emph{Insurance: Mathematics and Economics}, 106:\penalty0 90--114,
  2022.

\bibitem[Marotta et~al.(2017)Marotta, Martinelli, Nanni, Orlando, and
  Yautsiukhin]{marotta2017cyber}
A.~Marotta, F.~Martinelli, S.~Nanni, A.~Orlando, and A.~Yautsiukhin.
\newblock Cyber-insurance survey.
\newblock \emph{Computer Science Review}, 24:\penalty0 35--61, 2017.

\bibitem[Morgan(2020)]{morgan2020cybercrimecost}
S.~Morgan.
\newblock Cybercrime to cost the world \$10.5 trillion annually by 2025.
\newblock
  \url{https://cybersecurityventures.com/cybercrime-damage-costs-10-trillion-by-2025/},
  Nov. 2020.
\newblock Accessed: 2021-02-04.

\bibitem[Moscadelli(2004)]{moscadelli2004modelling}
M.~Moscadelli.
\newblock The modelling of operational risk: experience with the analysis of
  the data collected by the {B}asel {C}ommittee.
\newblock \emph{Available at SSRN 557214}, 2004.

\bibitem[Moumeesri and Pongsart(2022)]{moumeesri2022bonus}
A.~Moumeesri and T.~Pongsart.
\newblock Bonus-malus premiums based on claim frequency and the size of claims.
\newblock \emph{Risks}, 10\penalty0 (9):\penalty0 181, 2022.

\bibitem[Neuhaus(1988)]{neuhaus1988bonus}
W.~Neuhaus.
\newblock A bonus—malus system in automobile insurance.
\newblock \emph{Insurance: Mathematics and Economics}, 7\penalty0 (2):\penalty0
  103--112, 1988.

\bibitem[Nurse et~al.(2020)Nurse, Axon, Erola, Agrafiotis, Goldsmith, and
  Creese]{nurse2020data}
J.~R. Nurse, L.~Axon, A.~Erola, I.~Agrafiotis, M.~Goldsmith, and S.~Creese.
\newblock The data that drives cyber insurance: A study into the underwriting
  and claims processes.
\newblock In \emph{2020 International Conference on Cyber Situational
  Awareness, Data Analytics and Assessment (CyberSA)}, pages 1--8, 2020.

\bibitem[Oughton et~al.(2019)Oughton, Ralph, Pant, Leverett, Copic, Thacker,
  Dada, Ruffle, Tuveson, and Hall]{oughton2019stochastic}
E.~J. Oughton, D.~Ralph, R.~Pant, E.~Leverett, J.~Copic, S.~Thacker, R.~Dada,
  S.~Ruffle, M.~Tuveson, and J.~W. Hall.
\newblock Stochastic counterfactual risk analysis for the vulnerability
  assessment of cyber-physical attacks on electricity distribution
  infrastructure networks.
\newblock \emph{Risk Analysis}, 39\penalty0 (9):\penalty0 2012--2031, 2019.

\bibitem[Pal and Golubchik(2010)]{pal2010analyzing}
R.~Pal and L.~Golubchik.
\newblock Analyzing self-defense investments in internet security under
  cyber-insurance coverage.
\newblock In \emph{2010 IEEE 30th International Conference on Distributed
  Computing Systems}, pages 339--347, 2010.

\bibitem[Pal et~al.(2014)Pal, Golubchik, Psounis, and Hui]{pal2014will}
R.~Pal, L.~Golubchik, K.~Psounis, and P.~Hui.
\newblock Will cyber-insurance improve network security? {A} market analysis.
\newblock In \emph{IEEE INFOCOM 2014 - IEEE Conference on Computer
  Communications}, pages 235--243, 2014.

\bibitem[Pal et~al.(2019)Pal, Golubchik, Psounis, and Hui]{pal2017security}
R.~Pal, L.~Golubchik, K.~Psounis, and P.~Hui.
\newblock Security pricing as enabler of cyber-insurance a first look at
  differentiated pricing markets.
\newblock \emph{IEEE Transactions on Dependable and Secure Computing},
  16\penalty0 (2):\penalty0 358--372, 2019.

\bibitem[Pat{\'e}-Cornell et~al.(2018)Pat{\'e}-Cornell, Kuypers, Smith, and
  Keller]{pate2018cyber}
M.~Pat{\'e}-Cornell, M.~Kuypers, M.~Smith, and P.~Keller.
\newblock Cyber risk management for critical infrastructure: a risk analysis
  model and three case studies.
\newblock \emph{Risk Analysis}, 38\penalty0 (2):\penalty0 226--241, 2018.

\bibitem[Peters and Shevchenko(2015)]{peters2015advances}
G.~W. Peters and P.~V. Shevchenko.
\newblock \emph{Advances in heavy tailed risk modeling}.
\newblock Wiley Handbook in Financial Engineering and Econometrics. John Wiley
  \& Sons, Inc., Hoboken, NJ, 2015.
\newblock A handbook of operational risk.

\bibitem[Peters and Sisson(2006)]{peters2006bayesian}
G.~W. Peters and S.~A. Sisson.
\newblock Bayesian inference, {M}onte {C}arlo sampling and operational risk.
\newblock \emph{Journal of Operational Risk}, 1\penalty0 (3):\penalty0 27--50,
  Dec. 2006.

\bibitem[Peters et~al.(2011)Peters, Byrnes, and Shevchenko]{peters2011impact}
G.~W. Peters, A.~D. Byrnes, and P.~V. Shevchenko.
\newblock Impact of insurance for operational risk: is it worthwhile to insure
  or be insured for severe losses?
\newblock \emph{Insurance Math. Econom.}, 48\penalty0 (2):\penalty0 287--303,
  2011.

\bibitem[Peters et~al.(2016)Peters, Chen, and Gerlach]{peters2016estimating}
G.~W. Peters, W.~Y. Chen, and R.~H. Gerlach.
\newblock Estimating quantile families of loss distributions for non-life
  insurance modelling via {L}-moments.
\newblock \emph{Risks}, 4\penalty0 (2), 2016.

\bibitem[Peters et~al.(2018{\natexlab{a}})Peters, Shevchenko, and
  Cohen]{peters2017statistical}
G.~W. Peters, P.~V. Shevchenko, and R.~D. Cohen.
\newblock Statistical machine learning analysis of cyber risk data: event case
  studies.
\newblock In D.~Maurice, J.~Freund, and D.~Fairman, editors, \emph{FinTech:
  Growth and Deregulation}, chapter~3. Risk Books, 2018{\natexlab{a}}.

\bibitem[Peters et~al.(2018{\natexlab{b}})Peters, Shevchenko, and
  Cohen]{peters2018understanding}
G.~W. Peters, P.~V. Shevchenko, and R.~D. Cohen.
\newblock Understanding cyber-risk and cyber-insurance.
\newblock In D.~Maurice, J.~Freund, and D.~Fairman, editors, \emph{FinTech:
  Growth and Deregulation}, chapter~12. Risk Books, 2018{\natexlab{b}}.

\bibitem[Peters et~al.(2023)Peters, Malavasi, Sofronov, Shevchenko, Tr{\"u}ck,
  and Jang]{peters2023cyber}
G.~W. Peters, M.~Malavasi, G.~Sofronov, P.~V. Shevchenko, S.~Tr{\"u}ck, and
  J.~Jang.
\newblock Cyber loss model risk translates to premium mispricing and risk
  sensitivity.
\newblock \emph{The Geneva Papers on Risk and Insurance - Issues and Practice},
  48\penalty0 (2):\penalty0 372--433, 2023.

\bibitem[Ragulina(2017)]{ragulina2017bonus}
O.~Ragulina.
\newblock Bonus--malus systems with different claim types and varying
  deductibles.
\newblock \emph{Modern Stochastics: Theory and Applications}, 4\penalty0
  (2):\penalty0 141--159, 2017.

\bibitem[Rid and McBurney(2012)]{rid2012cyber}
T.~Rid and P.~McBurney.
\newblock Cyber-weapons.
\newblock \emph{The RUSI Journal}, 157\penalty0 (1):\penalty0 6--13, 2012.

\bibitem[Romanosky et~al.(2019)Romanosky, Ablon, Kuehn, and
  Jones]{romannosky2019content}
S.~Romanosky, L.~Ablon, A.~Kuehn, and T.~Jones.
\newblock {Content analysis of cyber insurance policies: how do carriers price
  cyber risk?}
\newblock \emph{Journal of Cybersecurity}, 5\penalty0 (1), 02 2019.
\newblock tyz002.

\bibitem[Schwartz and Sastry(2014)]{schwartz2014cyber}
G.~A. Schwartz and S.~S. Sastry.
\newblock Cyber-insurance framework for large scale interdependent networks.
\newblock In \emph{Proceedings of the 3rd International Conference on High
  Confidence Networked Systems}, HiCoNS '14, page 145–154, New York, NY, USA,
  2014. Association for Computing Machinery.

\bibitem[Shetty et~al.(2010)Shetty, Schwartz, Felegyhazi, and
  Walrand]{shetty2010competitive}
N.~Shetty, G.~Schwartz, M.~Felegyhazi, and J.~Walrand.
\newblock Competitive cyber-insurance and internet security.
\newblock In T.~Moore, D.~Pym, and C.~Ioannidis, editors, \emph{Economics of
  Information Security and Privacy}, pages 229--247. Springer US, Boston, MA,
  2010.

\bibitem[Shevchenko and Peters(2013)]{shevchenko2013loss}
P.~Shevchenko and G.~Peters.
\newblock Loss distribution approach for operational risk capital modelling
  under {B}asel {II}: Combining different data sources for risk estimation.
\newblock \emph{Journal of Governance and Regulation}, 2:\penalty0 33--57,
  2013.

\bibitem[Shevchenko et~al.(2023)Shevchenko, Jang, Malavasi, Peters, Sofronov,
  and Tr{\"u}ck]{shevchenko2022nature}
P.~V. Shevchenko, J.~Jang, M.~Malavasi, G.~W. Peters, G.~Sofronov, and
  S.~Tr{\"u}ck.
\newblock The nature of losses from cyber-related events: Risk categories and
  business sectors.
\newblock \emph{Journal of Cybersecurity}, to appear, 2023.

\bibitem[Tailor and Patel(2017)]{tailor2017comprehensive}
J.~P. Tailor and A.~D. Patel.
\newblock A comprehensive survey: ransomware attacks prevention, monitoring and
  damage control.
\newblock \emph{International Journal of Scientific Research}, 4:\penalty0
  2321--2705, 06 2017.

\bibitem[Targino et~al.(2013)Targino, Peters, Sofronov, and
  Shevchenko]{targino2013optimal}
R.~S. Targino, G.~W. Peters, G.~Sofronov, and P.~V. Shevchenko.
\newblock Optimal insurance purchase strategies via optimal multiple stopping
  times.
\newblock \emph{Preprint, arXiv:1312.0424}, 2013.

\bibitem[Tukey(1977)]{tukey1977exploratory}
J.~W. Tukey.
\newblock \emph{Exploratory data analysis}, volume~2.
\newblock Reading, MA: Addison-Wesley, 1977.

\bibitem[Tzougas et~al.(2018)Tzougas, Vrontos, and Frangos]{tzougas2018bonus}
G.~Tzougas, S.~Vrontos, and N.~Frangos.
\newblock Bonus-malus systems with two-component mixture models arising from
  different parametric families.
\newblock \emph{North American Actuarial Journal}, 22\penalty0 (1):\penalty0
  55--91, 2018.

\bibitem[von Skarczinski et~al.(2022)von Skarczinski, Dreissigacker, and
  Teuteberg]{skarczinski2022more}
B.~S. von Skarczinski, A.~Dreissigacker, and F.~Teuteberg.
\newblock More security, less harm? {E}xploring the link between security
  measures and direct costs of cyber incidents within firms using {PLS-PM}.
\newblock In \emph{Wirtschaftsinformatik 2022 Proceedings}, 2022.
\newblock URL \url{https://aisel.aisnet.org/wi2022/it_strategy/it_strategy/2}.

\bibitem[Wheatley et~al.(2016)Wheatley, Maillart, and
  Sornette]{wheatley2016extreme}
S.~Wheatley, T.~Maillart, and D.~Sornette.
\newblock The extreme risk of personal data breaches and the erosion of
  privacy.
\newblock \emph{The European Physical Journal B}, 89\penalty0 (1):\penalty0
  1--12, 2016.

\bibitem[Woods and B{\"o}hme(2021)]{woods2021sok}
D.~W. Woods and R.~B{\"o}hme.
\newblock {SoK}: Quantifying cyber risk.
\newblock In \emph{2021 IEEE Symposium on Security and Privacy (SP)}, pages
  211--228, 2021.

\bibitem[Xu and Hua(2019)]{xu2019cybersecurity}
M.~Xu and L.~Hua.
\newblock Cybersecurity insurance: Modeling and pricing.
\newblock \emph{North American Actuarial Journal}, 23\penalty0 (2):\penalty0
  220--249, 2019.

\bibitem[Xu et~al.(2014)Xu, Iglewicz, and Chervoneva]{xu2014robust}
Y.~Xu, B.~Iglewicz, and I.~Chervoneva.
\newblock Robust estimation of the parameters of g-and-h distributions, with
  applications to outlier detection.
\newblock \emph{Computational Statistics \& Data Analysis}, 75:\penalty0 66 --
  80, 2014.

\bibitem[Yang and Lui(2014)]{yang2014security}
Z.~Yang and J.~C. Lui.
\newblock Security adoption and influence of cyber-insurance markets in
  heterogeneous networks.
\newblock \emph{Performance Evaluation}, 74:\penalty0 1 -- 17, 2014.

\bibitem[Young and Yung(1996)]{young1996cryptovirology}
A.~Young and M.~Yung.
\newblock Cryptovirology: extortion-based security threats and countermeasures.
\newblock In \emph{Proceedings 1996 IEEE Symposium on Security and Privacy},
  pages 129--140, 1996.

\bibitem[Z{\"a}ngerle and Schiereck(2023)]{zangerle2023modelling}
D.~Z{\"a}ngerle and D.~Schiereck.
\newblock Modelling and predicting enterprise-level cyber risks in the context
  of sparse data availability.
\newblock \emph{The Geneva Papers on Risk and Insurance - Issues and Practice},
  48\penalty0 (2):\penalty0 434--462, 2023.

\bibitem[Zeller and Scherer(2022)]{zeller2022comprehensive}
G.~Zeller and M.~Scherer.
\newblock A comprehensive model for cyber risk based on marked point processes
  and its application to insurance.
\newblock \emph{European Actuarial Journal}, 12\penalty0 (1):\penalty0 33--85,
  2022.

\bibitem[Zhang and Zhu(2018)]{zhang2018optimal}
R.~Zhang and Q.~Zhu.
\newblock Optimal cyber-insurance contract design for dynamic risk management
  and mitigation.
\newblock \emph{Preprint arXiv:1804.00998}, 2018.

\bibitem[Zhang et~al.(2017)Zhang, Zhu, and Hayel]{zhang2017bi}
R.~Zhang, Q.~Zhu, and Y.~Hayel.
\newblock A bi-level game approach to attack-aware cyber insurance of computer
  networks.
\newblock \emph{IEEE Journal on Selected Areas in Communications}, 35\penalty0
  (3):\penalty0 779--794, 2017.

\end{thebibliography}

\end{document}